\documentclass[preprint]{imsart}

\RequirePackage[OT1]{fontenc}
\RequirePackage{amsthm,amsmath}
\RequirePackage[colorlinks,citecolor=blue,urlcolor=blue]{hyperref}
\usepackage{amssymb, bm, dsfont, mathtools, soul, tikz}

\startlocaldefs
\numberwithin{equation}{section}
\theoremstyle{plain}
\newtheorem{thm}{Theorem}[section]

\newtheorem{prop}[thm]{Proposition}
\newtheorem{rmq}[thm]{Remark}

\newtheorem{lem}[thm]{Lemma}

\newcommand{\zero}0
\newcommand{\tmpsum}{s}
\newcommand{\mix}{\text{\textnormal{mix}}}
\newcommand{\hit}{\text{\textnormal{hit}}}
\renewcommand{\d}{{\rm{d}}}

\newcommand{\ind}{\mathds{1}}

\newcommand{\reels}{\mathds{R}}

\newcommand{\cesp}[2]{\mathds{E}_{#1}\left[#2\right]}
\newcommand{\norm}[2]{\left\lVert #1\right\rVert_{#2}}

\newcommand{\R}{\mathbb{R}}
\newcommand{\N}{\mathbb{N}}

\renewcommand{\P}{\mathbb{P}}
\newcommand{\E}{\mathbb{E}}
\endlocaldefs

\begin{document}

\begin{frontmatter}
\title{Induced idleness leads to deterministic heavy traffic limits
for queue-based random-access algorithms}
\runtitle{Deterministic heavy traffic limits in QB-CSMA}

\begin{aug}
 \author{\fnms{Eyal} \snm{Castiel}\corref{}\thanksref{t2}\ead[label=e1]{eyal.castiel@isae-supaero.fr}},
 \author{\fnms{Sem} \snm{Borst}\ead[label=e2]{s.c.borst@tue.nl}},
 \author{\fnms{Laurent} \snm{Miclo}\ead[label=e3]{laurent.miclo@math.univ-toulouse.fr}},
 \author{\fnms{Florian} \snm{Simatos}\ead[label=e4]{florian.simatos@isae-supaero.fr}},
 \and
 \author{\fnms{Phil} \snm{Whiting} \ead[label=e5]{philip.whiting@mq.edu.au}}

 \thankstext{t2}{Authors are listed in alphabetical order except for the first one who is the main contributor.}

 \runauthor{E.\ Castiel et al.}

 \affiliation{ISAE-SUPAERO, Universit\'e de Toulouse, CNRS, Toulouse School of Economics, Eindhoven University of Technology, Nokia Bell Labs and Macquarie University}

 \address{ISAE-SUPAERO\\10 avenue Edouard Belin\\31055 Toulouse\\ 
 \printead{e1,e4}}

 \address{Eindhoven University of Technology\\
 \printead{e2}}

 \address{Toulouse School of Economics\\
 \printead{e3}}

 \address{Macquarie University\\
 \printead{e5}}

\end{aug}

\begin{abstract}
We examine a queue-based random-access algorithm where activation
and deactivation rates are adapted as functions of queue lengths.
We establish its heavy traffic behavior on a complete interference graph,
which turns out to be nonstandard in two respects:
(1) the scaling depends on some parameter of the algorithm and is not
the $N/N^2$ scaling usually found in functional central limit theorems;
(2) the heavy traffic limit is deterministic.
We discuss how this nonstandard behavior arises from the idleness
induced by the distributed nature of the algorithm.
In order to prove our main result, we develop a new method
for obtaining a fully coupled stochastic averaging principle.
\end{abstract}

\begin{keyword}[class=MSC]
\kwd[Primary ]{60K25}
\kwd[; secondary ]{60K35}
\end{keyword}

\begin{keyword}
\kwd{CSMA algorithms}
\kwd{stochastic averaging principle}
\kwd{state space collapse}
\end{keyword}

\end{frontmatter}

\setcounter{tocdepth}{1}
\tableofcontents

\newpage

\color{black}

\section{Introduction}

In the present paper we investigate the heavy traffic behavior
of a queue-based random-access mechanism.
Specifically, we analyze the joint queue length process
in a critically loaded system where packets arrive at the various nodes
as Poisson processes and are transmitted intermittently.
When all nodes are inactive, any of them may start a packet
transmission at an exponential rate that depends on its local queue
length, i.e., the number of packets pending for transmission.
Once a node is transmitting, it prevents other nodes from activating and turns inactive at an exponential rate
which is also governed by its local queue length.

\subsection{Context and motivation}

The above model arises in the context of distributed scheduling, i.e., deciding which queues to serve without any central authority having a global knowledge of the network state,
in queueing networks with constraints on the set of queues that can be active simultaneously (called constrained queueing networks). This is a fundamental
and challenging problem with applications in a wide range of settings.
In particular, the random-access mechanism described above captures
the dynamics of queue-based versions of the Carrier-Sense Multiple-Access
Collision Avoidance (CSMA-CA) protocol as further explained below,
which is commonly used in wireless communication networks.

A breakthrough in the area of scheduling in constrained queueing
networks was achieved when Tassiulas and Ephremides introduced the
Max-Weight algorithm in the early nineties~\cite{Tass90}.
This scheme was the first to provably offer maximum stability
guarantees under fairly general conditions, and has been generalized
and refined in a huge body of follow-up work.
However, the Max-Weight algorithm is inherently centralized in nature,
and crucially relies on the solution of a potentially NP-hard global
optimization problem at each iteration, namely, finding an independent set of maximum weight which then serves as schedule. This severely limits its
implementation in large-scale networks.

Only after nearly twenty further years, Jiang
and Walrand~\cite{Jiang08} and Rajagopolan {\em et al.}~\cite{Raj09}
proposed the first truly distributed algorithms with the capability
to match the throughput optimality of the Max-Weight algorithm.
Informally stated, these algorithms aim to mimic the scheduling operations
of the Max-Weight algorithm while using only locally available information.
Specifically, the individual nodes make fairly autonomous decisions
for controlling activity periods (e.g.~packet transmissions)
and inactivity periods (e.g.~back-off intervals), subject to the
constraints on simultaneous activity.

For a more detailed description, it is convenient to assume that the
latter constraints can be represented in terms of a conflict graph,
where the edges indicate which pairs of nodes are prevented
from simultaneous activity.
Such constraints may for example arise from interference issues
preventing simultaneous transmission in wireless networks, in which case
the conflict graph is commonly referred to as interference graph.
The operations of distributed scheduling algorithms in these scenarios
may be described as follows.
Upon completion of a packet transmission, a node either starts a random
back-off period or proceeds with the transmission of the next packet,
if any, with a probability that depends on the local queue length.
When inactive, a node simply runs down its back-off clock, but freezes
it whenever any of its neigbhors in the conflict graph are active,
ensuring that a back-off period can only end when all its neighbors
are inactive.
At that point, a node either initiates a packet transmission
or proceeds with the next random back-off period with a probability
which is also a function of the local queue length.

Now observe that transmission activity in wireless networks can be
detected by `sensing' a shared channel, and that the above back-off
mechanism precludes concurrent transmissions of mutually interfering nodes,
explaining the term Carrier-Sense Multiple-Access Collision Avoidance.
Further note that the idleness and randomized deactivation may seem
inefficient from a resource utilization perspective, but play
an instrumental role in sharing the medium through `listening'
in the absence of any centralized access control mechanism. The extreme case where a node deactivates only when its queue is empty is referred to as the Random Capture algorithm~\cite{FPR10}. Although it may seem to minimize idleness, it was actually shown that it is not always throughput-optimal~\cite{Ghaderi2012}.

When we now assume the interference graph to be complete and further
suppose that the back-off periods and transmission times are all
independent and exponentially distributed, the above model reduces
to that described in the first paragraph:
when all nodes are inactive, any of them may turn active
at a queue-dependent exponential rate, and once a node is active,
it de-activates at a queue-dependent exponential rate.

The seminal results in \cite{shashin,SST11} showed that the
above-described queue-based versions of the CSMA-CA algorithm
(henceforth referred to as QB-CSMA) achieve maximum stability,
provided that the activation and deactivation probabilities are
governed by suitable functions of the local queue lengths.
To the best of our knowledge, however, little is known about the queueing
dynamics of these algorithms beyond the maximum stability properties.


\subsection{A nonstandard heavy traffic behavior}

The present paper aims at deepening the understanding of the
above-described queue-based random-access algorithms.
We prove in particular that, near criticality and for the particular
case of a complete interference graph, these scheduling mechanisms
exhibit a nonstandard behavior in two ways:
\begin{enumerate}
\item the heavy traffic limit is deterministic;
\item the scaling depends on a parameter of the algorithm and is not
the usual central limit theorem (CLT) like scaling $N/N^2$ where the time-scale $N^2$ is the square of the space scale $N$.
\end{enumerate}

In particular, the limit is not a reflected Brownian motion and is thus
unconventional in the terminology of Harrison~\cite{Harrison95:0}. The literature on unconventional heavy traffic results is quite scarce, at least compared to the important number of conventional results. Harrison and Williams~\cite{Harrison96:0} exhibited the first such example in the context of a closed queueing network, and Kruk later provided an example for an open queueing network under the Earliest Deadline First policy~\cite{Kruk11:0}. Atar and Cohen~\cite{Atar19:1} study a multiclass
single-server queue which, subject to the usual CLT scaling, converges
to a nonstandard diffusion process (namely, a Walsh Brownian motion).
Another example is Puha~\cite{Puha15:0}, who studies the Shortest Remaining Processing Time (SRPT) policy:
there the scaling is nonstandard but the limiting diffusion is
conventional, i.e., the heavy traffic limit is a reflected Brownian motion.

In the model studied in the present paper, the behavior is nonstandard in two ways:
(1) the limit is actually deterministic and governed by an ordinary
differential equation (ODE), and (2) if $N$ is the space scale, the suitable
time scale is $N^{1+a}$ with $a \in (0,1/2)$ a parameter of the algorithm.
In particular, the time scale is in-between the usual fluid
and diffusion time scales $N$ and~$N^2$, respectively.
This peculiar scaling is due to the idleness which arises
as a consequence of the distributed nature of QB-CSMA, see Section~\ref{subsub:nonstandard} and~\ref{sub:other-scalings} for an illustrative back-of-the-envelope computation.

Despite these two nonstandard features, our model does exhibit a state
space collapse property which is commonly associated with the
conventional Brownian diffusion limits for a wide range
of multi-dimensional queueing processes \cite{Reiman84,Reiman05}.
Indeed, the queue lengths at the various nodes vary according
to certain fixed proportions in the heavy traffic limit, meaning that
the joint queue length process lives in a one-dimensional space.

\subsection{Idleness in random-access settings}

As alluded to above, in QB-CSMA, nodes deactivate at a state-dependent rate in order for the system to be able to alternate between different activity states in a distributed way. In particular, nodes may deactivate even when they have work to process. This makes the system non work-conserving and induces additional idleness compared to that owing to queues being empty.

\color{black}

For classical queueing models and stochastic networks, \textcolor{black}{a large body of literature has investigated the impact of idleness on the heavy traffic behavior and performance. For instance, one of the achievements in the study of Jackson networks is the understanding of this impact on the reflection matrix in the limiting multi-dimensional reflected Brownian motion~\cite{Harrison81:1}.} However, in these ``classical'' settings, idleness occurs when queues
are empty or resources get stranded because of concurrency requirements.

In contrast, in random-access settings like ours, idleness occurs even when there are large queues, and is simply part of a distributed mechanism to share resources without explicit information exchange.
In this distributed setting, the impact of idleness on heavy traffic behavior is more subtle and model-dependent.
For instance, considering QB-CSMA in a different regime than the one studied here, a lingering effect was highlighted in~\cite{Sim14} leading to a heavy traffic scaling $\frac{1}{(1-\rho)^2}$, \textcolor{black}{with $\rho$ the load of the network,} compared to the usual $\frac{1}{1-\rho}$ due to idleness. 
In the present model, the fraction of idleness is inversely proportional to (a power of) the queue lengths, yielding a yet different impact on the heavy traffic behavior.
After the model and main results are presented, we will describe this behavior in greater detail in Section~\ref{subsub:nonstandard}, and in particular explain why the heavy traffic behavior is deterministic and the $N/N^{1+a}$ scaling emerges.

It is interesting to compare our results with those on Max-Weight. Indeed, QB-CSMA algorithms were designed with the purpose of mimicking Max-Weight in a decentralized manner, and Shah and Shin~\cite{shashin} establish the throughput-optimality of these algorithms by applying the same Lyapunov function as for Max-Weight. Thus, as far as throughput is concerned, QB-CSMA algorithms behave very similarly as Max-Weight. What we show here is that the comparison breaks down at criticality concerning delay. Indeed, Stolyar~\cite{Sto04} showed that the critical behavior of Max-Weight is ``standard'', i.e., consists in the usual CLT scaling and leads to a reflected Brownian motion. Here the behavior is completely different because of the additional idleness induced by the decentralized nature of QB-CSMA.

\subsection{Link with polling systems}

When run on a complete interference graph, only one server can be active
at a time and so QB-CSMA can be viewed as a particular polling system
with state-dependent non-zero switchover times and switching decisions.
This equivalence has in fact been exploited to use results
for polling systems with a so-called $1$-limited service discipline
and a probabilistic routing policy in analyzing CSMA algorithms
where nodes deactivate at a fixed (non-queue-based) rate,
see for instance~\cite{Cecchi16:0,Dorsman15:0}.

There is a significant body of heavy traffic results for polling
systems by now, starting with the seminal papers~\cite{Coffman95:0, Coffman98:0}.
However, the model in~\cite{Coffman95:0} did not include any switchover times,
so that the total amount of work behaves as in a work-conserving
single-server queue and in particular exhibits the standard
heavy traffic scaling behavior.
The model in~\cite{Coffman98:0} did incorporate non-zero switchover times,
but involved an exhaustive service discipline, which implies that the
fraction of idleness is basically reciprocal to the queue length,
rather than the queue length raised to a power $a \in (0, 1/2)$.
While the total amount of work is substantially larger than in
a work-conserving single-server queue due to the non-zero switchover times,
it exhibits a similar $\frac{1}{1-\rho}$ scaling behavior because
of the rapid decay of the idleness as function of the queue length.
Moreover, the exhaustive service discipline causes the work to rapidly
shift among the various queues, causing fundamentally different
dynamics than the state space collapse that we observe in our model.

Heavy traffic results for a broader class of polling systems
with so-called Bernoulli-exhaustive and Bernoulli-gated service
disciplines are established in~\cite{Mei07:0}.
However, these concern stationary distributions rather than
process-level limits, and again pertain to disciplines where the
idleness scales inversely proportional to the queue length,
yielding qualitatively similar scaling behavior as in~\cite{Coffman98:0}.

Finally, heavy traffic results for polling systems with $k$-limited
service disciplines are presented in~\cite{Boon14:0}.
In these systems the idleness essentially approaches a constant,
positive fraction as the queue lengths grow, again causing fundamentally
different scaling behavior from what we encounter in our model.

\subsection{Methodological contribution}

The seemingly simple case of a complete interference graph actually turns out to be challenging to analyze. Technically, the main difficulty lies in controlling the so-called stochastic averaging principle, or homogenization. This principle asserts that when two processes interact but evolve on different time scales, then the 'slow' process only interacts with the 'fast' process through the instantaneous equilibrium distribution of the fast process. The most difficult case is the so-called fully coupled stochastic averaging principle which arises when this instantaneous distribution depends on the state of the slow process, which is the case here.

Controlling such an approximation is in general a difficult problem, and numerous methods have been developed for that, see for instance the classical monograph of Freidlin and Wentzell~\cite{Frei98}. However, in our case we were not able to apply any standard method, in particular the ones developed by Kurtz~\cite{Kurtz92} and Luczak and Norris~\cite{Luc13}. This led us to develop a new method. It is close in spirit to that of Luczak and Norris but is more tailored to Markov processes. The stochastic averaging principle is controlled by martingale arguments and leverages properties of solutions to the Poisson equations associated with the fast generators, see Section~\ref{sub:SAP} for more details. We believe that this new approach has the potential of being applied to a wide class of problems and its more general applicability will be studied elsewhere.

\color{black}

To give a more precise idea of our technique to control the homogenization, imagine the Markov process under study is $(Q^N, \sigma^N)$ with $Q^N$ the 'slow' process and $\sigma^N$ the 'fast' one. Controlling homogenization amounts to controlling the approximation
\[ \int_0^t F \left( Q^N(s), \sigma^N(s) \right) \d s \approx \int_0^t \pi^{Q^N(s)} \left[ F \left( Q^N(s), \cdot \right) \right] \d s \]
with $\pi^q$ the stationary distribution of the fast process when the slow process is in state $q$ and $\nu[f]$ the integral of a measurable function $f$ with respect to the measure $\nu$. To do so, we first rewrite the difference
\[ \int_0^t \left( F \left( Q^N(s), \sigma^N(s) \right) - \pi^{Q^N(s)} \left[ F \left( Q^N(s), \cdot \right) \right] \right) \d s \]
in the form
\begin{multline*}
	V(Q^N(t), \sigma^N(t)) - V(Q^N(0), \sigma^N(0))\\
	- \int_0^t L^{N, \sigma^N(s)}_\textrm{s} (V(\, \cdot \, , \sigma^N(s))) (Q^N(s)) \d s + \textrm{(martingale term)}
\end{multline*}
with $L^{N, \sigma}_\textrm{s}$ the generator of the slow process when the slow process is in state~$\sigma$. Here the function $V$ that appears is linked to solutions to the Poisson equation $L^{N, q}_\textrm{f}\phi = g - \pi^q[g]$ with unknown $\phi$, and so the above expression indeed makes it possible to cast the problem of homogenization in terms of control of solutions to Poisson equations. Moreover, this control is achieved by expressing solutions $\phi$ in the form
\[ \phi(\sigma) = \int_0^\infty \left[ \E_\sigma(g(X(t))) - \E_\sigma(g(X(\infty))) \right] \d t \]
with $(X(t))$ the fast Markov process started at $\sigma$ under $\P_\sigma$, and $X(\infty)$ is stationary distribution. To the best of our knowledge, this approach for controlling homogenization, and in particular the bounds on the solutions to the Poisson equation that we establish are new.

\subsection{Organization of the paper}

We introduce our model and state our main result in Section~\ref{sec:model+main}. This section also presents a discussion of the result, in particular why we consider polynomial activation functions, a back-of-the-envelope computation to provide an intuition for the result, and also a more detailed discussion of the stochastic averaging principle. Section~\ref{sec:notation} gathers the notation used throughout the paper, and in particular the generators and their associated Poisson equations as well as important stopping times used in localization arguments. The three main steps of the proof are then presented. Sections~\ref{homo} to \ref{Smain} contain the technical arguments: Sections~\ref{homo} and~\ref{sec:Poisson} describe the arguments controlling the stochastic averaging principle, Section~\ref{Sssc} the arguments controlling the state space collapse, and Section~\ref{Smain} gathers the arguments to provide the full proof. The paper is concluded with different extensions and directions for future research in Section~\ref{sec:openings}.

\color{black}

\section{Model description and main result} \label{sec:model+main}

\subsection{Model description with fixed arrival rates}

We have a set of $n$ nodes labeled by $V = \{1, \ldots, n\}$. Each node $v \in V$ represents an $M/M/1$ queue with the FIFO service discipline and vacations, its arrival rate is denoted by $\lambda_v > 0$. We denote by $Q_v(t) \in \mathds{N} \coloneqq \{0, 1, \ldots, \}$ the length of $v$'s backlog at time $t$ and by $\sigma_v(t) \in \{0,1\}$ the activity process: the server at $v$ is active and processing pending requests at unit rate if $\sigma_v(t) = 1$, and $\sigma_v(t) = 0$ otherwise. Put differently, $\sigma_v(t)$ is the instantaneous service rate of node $v$ at time $t$. We define $\lambda \coloneqq (\lambda_v, v \in V)$, $Q(t) \coloneqq (Q_v(t), v \in V)$ and $\sigma(t) \coloneqq (\sigma_v(t), v \in V)$.

We impose that only one node can be active at a time, and so whenever convenient we will identify $\sigma$ with the active node, or put $\sigma = 0$ if no node is active (empty schedule). We will thus either consider $\sigma \in \{0,1\}^V$ when seeing $\sigma$ as the vector of instantaneous service rates, or $\sigma \in V_0$ with $V_0 = V \cup \{0\}$ when seeing $\sigma$ as the current schedule. Because a schedule is associated with a node, we will sometimes use the notation $q_\sigma$ to denote the $v$th coordinate of the vector $q \in \mathds{R}^V_+$, with $v$ the only non-zero coordinate of~$\sigma$, and in this case we will adopt the convention $q_0 = 0$. Note that with this convention, we have $\sigma_0 = 1 - \sum_{v \in V} \sigma_v$.
\\

Given the current schedule $\sigma$, the queue-length process $Q$ evolves as $n$ independent $M/M/1$ queues with service rates $\sigma$ and input rates $\lambda$. On the other hand, given the current value $Q$ of the queue-length process, $\sigma$ evolves according to the \textcolor{black}{following dynamic, which is a particular case of the Glauber dynamics for the hard-core model~\cite{Dobrushin68:0, Berg94:0}:} an active node $v$ with $\sigma_v = 1$ deactivates at rate $\Psi_-(Q_v)$ for some deactivation function $\Psi_-$, and an inactive node $v$ with $\sigma_v = 0$ activates at rate $\Psi_+(Q_v)$ for some activation function $\Psi_+$, provided no other node is active.

To be more formal, $(Q, \sigma)$ is a Markov process on $\mathds{N}^V \times \{0,1\}^V$ with infinitesimal generator $L$ that can be decomposed as the sum of two generators:
\begin{itemize}
	\item the generator $L^\sigma_{\mathrm{s}}$ of the \textit{slow} queue-length process $Q$ whose dynamic depends on $\sigma$;
	\item and the generator $L^q_{\mathrm{f}}$ of the \textit{fast} activity process $\sigma$ whose dynamic depends on $q$.
\end{itemize}
The terminology \textit{slow} and \textit{fast} will be justified in Section~\ref{sub:SAP} when discussing the stochastic averaging principle. Thus, $L$ acts on functions $f: \mathds{N}^V \times \{0,1\}^V \to \mathds{R}$ as
\[ Lf(\sigma,q) = L_{\mathrm{s}}^\sigma(f(\sigma, \cdot))(q) + L_{\mathrm{f}}^q(f(\cdot,q))(\sigma) \]
with
\begin{equation} \label{eq:slow-L}
	L_{\mathrm{s}}^\sigma(g)(q) = \sum_{v \in V} \lambda_v \left( g(q+e^v) - g(q) \right) + \sum_{v \in V} \sigma_v \ind_{q_v>0} \left( g(q-e^v) - g(q) \right)
\end{equation}
and
\begin{multline} \label{eq:fast-L}
	L_{\mathrm{f}}^q(h)(\sigma) = \sum_{v \in V} \sigma_v \Psi_-(q_v) \left( h(\sigma - e^v) - h(\sigma) \right)\\
	+ \prod_{w \in V} (1-\sigma_w) \sum_{v \in V} \Psi_+(q_v) \left( h(\sigma + e^v) - h(\sigma) \right)
\end{multline}
with $g: \mathds{N}^V \to \mathds{R}$ and $h: \{0,1\}^V \to \mathds{R}$ arbitrary functions and $e^v \in \{0,1\}^V$ with $0$'s everywhere except at the $v$th coordinate equal to $1$. \textcolor{black}{Note that a server does not deactivate immediately when its queue gets empty, which makes the indicator term $\ind_{q_v > 0}$ necessary in~\eqref{eq:slow-L}.} Since the graph associated with $L_{\mathrm{f}}^q$ is a star centered at $0$, this generator admits a reversible distribution denoted $\pi^q$. For reasons explained in Section~\ref{sub:discussion}, we consider polynomial activation and deactivation functions of the form
\[ \Psi_+(x) = \frac{(x+1)^a}{1 + (x+1)^a} \in [0,1] \ \text{ and } \ \Psi_-(x) = 1 - \Psi_+(x), \ x \in \mathds{N}, \]
with $a > 0$ a parameter of the algorithm. In this case, $\pi^q$ is given by
\[ \pi^q(\sigma) = \frac{(1+q_\sigma)^a}{\sum_{\eta \in V_0} (1+q_\eta)^a}, \ \sigma \in V_0. \]

\color{black}
Let $\rho = \sum_{v \in V} \lambda_v$. Under the above assumptions, it is not hard to establish that $(Q, \sigma)$ is positive recurrent if $\rho < 1$ and transient if $\rho > 1$. Transience for $\rho > 1$ can be proved by lower bounding $X := \sum_v Q_v$ by an $M/M/1$ queue with arrival rate $\rho$ and service rate $1$. Positive recurrence for $\rho < 1$ can be proved using the Foster--Lyapunov criterion and showing that $X$ is a Lyapunov function. Indeed, as soon as one queue is active, arrivals make $X$ increase in the mean by $\rho$ while the queue in service makes it decrease by $1$, so that overall $X$ decreases at rate $\rho - 1 < 0$.

Thus, the regime where $\sum_v \lambda_v = 1$ will be referred to as the critical case and the rest of the paper will be devoted to the study of the near-critical case where $\sum_v \lambda_v \approx 1$, which we introduce now.

\subsection{Near-critical regime and heavy traffic scaling}

Throughout the paper, we fix $V = \{1, \ldots, n\}$, $a > 0$, $\lambda^\infty \in \R^V_+$ with $\sum_v \lambda^\infty_v = 1$ and $\gamma \in \R^V$. For each $\varepsilon > 0$, we define $N = \varepsilon^{-1/a}$ and consider
\begin{equation} \label{eq:near-criticality}
	\lambda^N = \lambda^\infty - \varepsilon \gamma = \lambda^\infty - N^{-a} \gamma.
\end{equation}
The parameter $\varepsilon > 0$ represents the 'distance' between $\lambda^N$ and the boundary of the stability region. We introduce $N = \varepsilon^{-1/a}$ because it will be simpler to index the processes by $N$ rather than by $\varepsilon$, as is for instance reflected by the notation $\lambda^N$ instead of $\lambda^\varepsilon$. As will be seen shortly, $N = \varepsilon^{-1/a}$ is the 'right' order of magnitude of the queue length process (see the discussion in Section~\ref{sub:other-scalings} for more details).


Our main object of interest is the Markov process with infinitesimal generator given by~\eqref{eq:slow-L} and~\eqref{eq:fast-L} but with $\lambda^N$ instead of $\lambda$. Thus, the generator $L^\sigma_{\mathrm{s}}$ that we will consider actually depends on $N$ and is given by
\[ L_{\mathrm{s}}^\sigma(g)(q) = \sum_{v \in V} \lambda^N_v \left( g(q+e^v) - g(q) \right) + \sum_{v \in V} \sigma_v \ind_{q_v>0} \left( g(q-e^v) - g(q) \right), \]
but in order to avoid cumbersome notation we will omit this dependency in $N$. Likewise, we will denote by $L$ the generator  $L = L^q_\mathrm{f} + L^\sigma_\mathrm{s}$ introduced above but with $\lambda^N$ instead of $\lambda$, and in the sequel we will denote by $(Q, \sigma)$ the Markov process with this generator (again, omitting the dependency in~$N$ for ease of notation).

In contrast, we will keep the dependency in $N$ for the scaled processes. More precisely, we consider $(Q^N, \sigma^N)$ the Markov process obtained from $(Q, \sigma)$ by speeding up time by a factor $N^{1+a}$ and scaling the $Q$-components by $N$ in space:
\[ Q^N(t) = \frac{1}{N} Q \left( N^{a+1} t \right) \ \text{ and } \ \sigma^N(t) = \sigma \left( N^{a+1} t \right), \ t \geq 0. \]
The infinitesimal generator of $(Q^N, \sigma^N)$ will be denoted by $L^N$, see Section~\ref{subsub:generators} for an explicit formula.

\begin{rmq}
	Other scalings are possible: actually, when the arrival rates are still given by~\eqref{eq:near-criticality} and $\varepsilon$ is the distance to the boundary of the stability region, we investigate in Section~\ref{sub:other-scalings} what happens on the space scale $\varepsilon^{-1/a'}$ with $a' > 0$ not necessarily equal to $a$.
\end{rmq}

\subsection{Main result}

For $x \in \mathds{R}^V$ and $b > 0$, let in the sequel
\[ \norm{x}{b}=\left( \sum_{v \in V} \lvert x_v \rvert^b\right)^{1/b} \ \text{ and } \ s(x) = \sum_{v \in V} x_v. \]
\color{black}
As will be seen shortly, the limiting process lives in the one-dimensional vector space
\begin{align} \label{eq:I}
	I & = \left\{ x \in \mathds{R}^V_+: \lambda_w^\infty x_v^a=\lambda_v^\infty x_w^a, \ v, w \in V \right\}\\
	& = \left\{ x \in \mathds{R}^V_+: x_v = \left( \frac{\lambda_v^{\infty}}{\mu} \right)^{1/a} \tmpsum(x), \ v \in V \right\}\notag,
\end{align}
\textcolor{black}{where here and in the sequel, $\mu = \norm{\lambda^\infty}{1/a}$. Intuitively, $I$ is the space where the mean service rate at each node matches the corresponding arrival rate.} In the sequel we use $\Rightarrow$ to denote weak convergence as $N \to \infty$. The following result is the main result of the paper, which describes the behavior of the queue-length process in the \textcolor{black}{near-critical case}.
\begin{thm}\label{main}
	Assume that the three following assumptions hold:
	\begin{itemize}
	\item $a < 1/2$;
	\item \textcolor{black}{condition~\eqref{eq:near-criticality} holds, i.e., $\lambda^N = \lambda^\infty - N^{-a} \gamma$ with $\lambda^\infty$ and $\gamma$ introduced above;}
	\item $Q^N(0) \Rightarrow q^0$ for some $q^0 \in I \setminus \{0\}$.
\end{itemize}
	Then $Q^N \Rightarrow q$ uniformly on compact time-sets, where $q$ is uniquely characterized as follows: $q(t) \in I$ for every $t \geq 0$ and $\tmpsum \circ q$ is the unique solution to the ODE
	\[ \textcolor{black}{\dot x = \mu x^{-a}-\tmpsum(\gamma)} \]
	with initial condition $x(0) = \tmpsum(q^0)$ and where $\mu = \textcolor{black}{\norm{\lambda^\infty}{1/a}}$.
\end{thm}

\color{black}
Except when $\tmpsum(\gamma) = 0$, there does not seem to be an explicit formula for the solution of the previous ODE. For $\tmpsum(\gamma) = 0$, the solution to the ODE $\dot x = \mu x^{-a}$ is
\[ x(t) = \left( x(0)^{a+1} + (a+1) \mu t \right)^{1/(a+1)}, \ t \geq 0, \]
and so using the fact that $q(t) \in I$, the limit $q$ in the previous statement is given in this case by
\[ q_v(t) = \left( \frac{\lambda^\infty_v}{\mu} \right)^{1/a} \left( \tmpsum(q^0)^{a+1} + (a+1) \mu t \right)^{1/(a+1)}, \ v \in V, t \geq 0.\]

\color{black}
From now on, we assume that the conditions of this theorem are enforced, i.e., we assume throughout that $a < 1/2$, that~\eqref{eq:near-criticality} holds and that $Q^N(0) \Rightarrow q^0 \in I \setminus \{0\}$.

With some extra work, but without giving much more insight on the system's behavior, the previous result could be generalized to an arbitrary initial condition $q^0 \in \mathds{R}^V_+$. If $q^0 = 0$ nothing changes in the statement of the above result, while if $q^0 \not \in I$ then the convergence holds uniformly on compact time-sets from $(0, +\infty)$ because the limiting process immediately jumps at time $0+$ to the invariant manifold $I$ even if it does not start there. The rest of this introduction is devoted to discussing this result in more details.

\subsection{Intuition and discussion} \label{sub:discussion} We discuss here in more details the context and implications of our result. We begin by justifying our interest in polynomial activation functions, then give an intuition behind the state space collapse result based on the stochastic averaging principle, and we finally discuss the nonstandard scaling that emerges from it.

\subsubsection{Polynomial activation functions} \label{subsub:csma} The literature on optimal CSMA algorithms is very rich and the interested reader is for instance referred to the thorough survey by Yun et al.~\cite{Yun12:0} for more details. In this paper we are interested in the class of QB-CSMA algorithms initially proposed by Rajagopalan, Shah and Shin~\cite{Raj09}. The main idea of these algorithms is to have activation and deactivation rates $\Psi_+$ and $\Psi_-$ being adapted as a function of queue lengths. Rajagopalan, Shah and Shin study in particular the case where $\Psi_+ + \Psi_- = 1$ with
\[ \Psi_+(q) = \frac{f(q_v)}{1 + f(q_v)} \]
for some function $f$. The main result of~\cite{ghad10, Raj09, shashin} is that this algorithm is throughput-optimal for any interference graph provided $f$ increases slowly enough, namely sub-polynomially\footnote{Actually, these algorithms also use some information on the current maximum queue length, whether the exact maximum or an estimation thereof.}. However, results of~\cite{Ghaderi2012} suggest that if $f$ grows polynomially, then it is only throughput-optimal for some interference graphs, depending on the relation between the graph topology and the exponent of the polynomial growth of $f$.

The rationale for seeking fast-increasing functions $f$ is that a folklore result has it that delay is improved with faster increasing functions $f$, an intuition which is backed up by results in~\cite{Bouman11:0}. Polynomial activation and deactivation functions should therefore achieve the optimal trade-off between throughput and delay for this class of algorithms, which is the reason why we focus on this case here. Note that in the case of a complete interference graph as considered here, the algorithm is throughput-optimal for any functions $\Psi_+$ and $\Psi_-$ satisfying $\Psi_+(q) \to 1$ and $\Psi_-(q) \to 0$ as $q \to \infty$, so that we need not worry about stability issues for such polynomial activation and deactivation functions, as may be the case in a more general setting.

\subsubsection{State space collapse from the stochastic averaging principle} \label{sub:SSC}

The reason behind the state space collapse property is simple to understand based on the \textit{stochastic averaging principle}. Put simply, when queue lengths are large, say of the order of $N$, then the typical time scale of $\sigma$ is much faster than the one of $Q$ which makes $Q$ interact with $\sigma$ only through the stationary distribution $\pi^q$ of its corresponding instantaneous Glauber dynamics. The latter depends on $Q$, which gives rise to a so-called fully coupled stochastic averaging principle, which essentially amounts to the approximation
\begin{equation} \label{eq:SAP}
	\int_0^t F \left( Q^N(s), \sigma^N(s) \right) \d s \approx \int_0^t \pi^{N Q^N(s)} \left[ F \left( Q^N(s), \cdot \right) \right] \d s
\end{equation}
with $\nu[f] = \int f \d \nu$ for any positive measure $\nu$ and integrable function $f$.

Recall that in our case, the stationary probability $\pi^q(v)$ of node $v \in V$ being active is given by
\[ \pi^q(v) = \frac{(1+q_v)^a}{1 + \sum_{w \in V} (1+q_w)^a}. \]
According to the stochastic averaging principle, this should represent the instantaneous service rate of node $v$ which should thus behave as a subcritical $M/M/1$ queue when $\pi^q(v) < \lambda^\infty_v$ and as a supercritical $M/M/1$ queue when $\pi^q(v) > \lambda^\infty_v$. As $\sum_v \lambda^\infty_v = 1$ and $\sum_{v \in V} \pi^q(v) = 1 - \pi^q(0) \approx 1$ for large $q$, we see that the only way for the network to behave smoothly is that each average service rate $\pi^q(v)$ matches its incoming service rate, i.e., $\pi^q(v) \approx \lambda^\infty_v$. When $q$ is large, this forces $q$ to live in the invariant manifold~$I$ because $\pi^q(v) \approx q^a_v$ up to a multiplicative constant. Thus, the state space collapse phenomenon can be directly understood as a consequence of the stochastic averaging principle together with the criticality assumption.

\subsubsection{The stochastic averaging principle} \label{sub:SAP}

In the context of stochastic networks, the \textit{stochastic averaging principle} was put forth for loss networks in the famous work by Hunt and Kurtz~\cite{Hunt94} but, as mentioned in Feuillet and Robert~\cite{Feu12}, ``outside this class of networks, there are, up to now, few examples of stochastic networks for which a fully coupled stochastic averaging principle occurs''. Establishing a fully coupled stochastic averaging principle is in general a challenging task and, in the queueing literature, many works actually restrict their study to the so-called homogenized process, assuming that timescale separation indeed occurs.

Rigorous proofs of stochastic averaging principles were established for polling systems times~\cite{Coffman95:0, Coffman98:0, Jennings10:0}, for models of distributed hash tables~\cite{Feu12} and for the $X$ model~\cite{Perry13:0}. Luczak and Norris~\cite{Luc13} also developed a new method which they applied to a variant of the supermarket model.

Most of these works, in particular~\cite{Feu12, Hunt94, Perry13:0}, rely on the machinery developed by Kurtz~\cite{Kurtz92}. It relies on martingale arguments and identifies the asymptotic occupation measure of the fast process as the invariant measure of a limiting averaged generator. In our case this identification step is not clear because some rates go to $0$ in the limit. In particular, the limiting scheduling process is degenerate: it starts at $0$ and then jumps to one of the possible states $v \in V$ where it is absorbed. In the absence of uniqueness, it is known that any accumulation point must be a linear combination of the different stationary measures but no general method seem to exist to characterize this combination.

The method of Luczak and Norris~\cite{Luc13} does not yield this problem. However, we have not been able to apply their results to our case. It seems plausible to modify their arguments in order to obtain Theorem~\ref{main} but only for $a < \frac{1}{3}$. The method that we develop here is close in spirit to theirs but is more tailored to Markov processes. The approximation~\eqref{eq:SAP} is controlled by martingale arguments and leverages properties of solutions to the Poisson equations (in $\phi$) $L^q_{\mathrm{f}} \phi = g - \pi^q[g]$ associated with the fast generators $L^q_{\mathrm{f}}$ and to functions $g: V_0 \to \R$.

\subsubsection{Nonstandard behavior} \label{subsub:nonstandard}

Taking the state space collapse and the stochastic averaging principle for granted, back-of-the-envelope computation can give insight into the nonstandard critical behavior observed for our system. As mentioned above, a consequence of the stochastic averaging and the criticality assumption is that $\pi^q(v) \approx \lambda^\infty_v$. However, taking into account the idle time induced by the necessary scheduling of the empty state which, when queue lengths are of the order of $N$, is of the order $\pi^q(0) \approx N^{-a}$, gives rise to the second-order approximation where $\lambda_v^\infty - \pi^q(v)$ is of the order of $N^{-a}$ (see Section~\ref{sub:other-scalings} for a more detailed heuristic). This suggests that node $v \in V$ behaves as a near-critical $M/M/1$ queue with arrival rate $\lambda^N_v = \lambda^\infty_v - N^{-a} \gamma_v$ and service rate $\lambda_v^\infty - N^{-a}$.

What is the right time scale for such a queue? A first-order asymptotic expansion of its generator can give a clue, namely, if time is sped up by $N^b$ then the action on its generator on a function $f$ is given by
\begin{multline*}
	N^b \left(\lambda^\infty_v -N^{-a} \gamma_v \right) \left( f \left( q + \frac{1}{N} \right) - f(q) \right)\\
	+ N^b \left( \lambda^\infty_v - N^{-a} \right) \left( f \left( q - \frac{1}{N} \right) - f(q) \right).
\end{multline*}
The leading term is $N^{b-a-1} f'(q)$ which suggests to take $b = a+1$, as turns out to be indeed the case. Moreover, we see that only first-order terms are dominant, which explains why the limiting process is deterministic and no diffusion term arises. This discussion also clearly highlights the key impact of idleness on the system performance at criticality, as without idleness, i.e., if we had $\lambda^\infty - \pi^q$ of the order of $1/N$, then we would see the usual $N/N^2$ scaling and a diffusion process in the limit.

\section{Notation and main steps of the proof} \label{sec:notation}

We introduce in this section further notation, and then explain the main steps of the proof of Theorem~\ref{main}.

\subsection{Notation} We first gather notation used throughout the paper.

\subsubsection{General notation}

For $b > 0$ and $x \in \R^V$ recall the notation $\norm{x}{b} = (\lvert x_1 \rvert^b + \cdots + \lvert x_n \rvert^b)^{1/b}$ and $s(x) = x_1 + \cdots + x_n$. We write $\lVert \cdot \rVert_\infty$ for the supremum norm, thus $\lVert f \rVert_\infty = \sup \lvert f \rvert$ for $f: \reels^V \to \reels$ and $\lVert q \rVert_\infty = \max_v \lvert q_v \rvert$ for $q \in \R^V$. If $U \subset \reels^V$ and $f: \reels^V \to \reels$ we also define $\lVert f \rVert_{U, \infty} = \sup_{x \in U} \lvert f(x) \rvert$.

Whenever $f$ is smooth enough, we denote by $\partial_v$ its partial derivative along $q_v$ and $\partial^2_{v,w}$ its second-order derivative along $q_v$ and $q_w$, i.e.,
\[ \partial_v f = \frac{\partial f}{\partial q_v} \ \text{ and } \ \partial^2_{v,w} f = \frac{\partial^2 f}{\partial q_v \partial q_w}. \]
We will also consider the discrete differences $\Delta^N_{\pm, v}f$ for a function $f: \R^V \times V_0 \to \R$, given by
\[ \Delta^N_{\pm, v}f(q, \sigma) = f \left( q \pm \frac{e^v}{N}, \sigma \right) - f(q, \sigma). \]
Thus, $N \Delta^N_{\pm, v}f \to \pm \partial_v f$ as $N \to \infty$ for $f$ differentiable.

\subsubsection{Generators} \label{subsub:generators}

Let $E^N = \frac{1}{N} \N^V$ be the state space of the scaled process $Q^N$. We define $L^N$, $L^{N,q}_{\mathrm{f}}$ and $L^{N, \sigma}_{\mathrm{s}}$ for $q \in E^N$ and $\sigma \in V_0$ the scaled generators \textcolor{black}{ with arrival rates $\lambda^N$}: for $f : E^N \times V_0 \to \R$, $g : E^N \to \R$, $h : V_0 \to \R$, $q \in E^N$ and $\sigma \in V_0$,
\[ L^Nf(q, \sigma) = N^{a+1} L f^{(N)}(Nq, \sigma), \ L^{N, \sigma}_{\mathrm{s}} g (q) = N^{a+1} L^\sigma_{\mathrm{s}} g^{(N)}(Nq) \]
and
\[ L^{N,q}_{\mathrm{f}} h (\sigma) = N^{a+1} L^{Nq}_{\mathrm{f}} h(\sigma) \]
with $f^{(N)}(q, \sigma) = f(q/N, \sigma)$ and $g^{(N)}(q) = g(q/N)$ for $q \in \N^V$. Note that the stationary distribution of $L^{N, q}_{\mathrm{f}}$ is $\pi^{Nq}$. Let $\Gamma^N$ be the \textit{carr\'e du champ} operator associated with $L^N$: for any $f: E^N \times V_0 \to \mathbb{R}$, .  We have
\[ \Gamma^N (f) = L^N (f^2) - 2 f L^N (f) \]
and elementary computation shows that for $(q, \sigma) \in E^N \times V_0$ we have
\begin{align} \label{eq:Gamma}
\Gamma^N & f(q,\sigma) = N^{a+1} \sum_{v \in V} \lambda^{N}_v \left( f \left( q+\frac{e^v}{N},\sigma \right)-f(q,\sigma)\right)^2\\
& \hspace{-5mm} + N^{a+1} \sum_{v \in V} \sigma_v\ind_{q_v>0} \left(f \left( q-\frac{e^v}{N},\sigma \right)-f(q,\sigma)\right)^2 \notag\\
& \hspace{-5mm} + N^{a+1} \sum_{v \in V} \left( f(q,0)-f(q, e^v) \right)^2 \left( \dfrac{\sigma_v}{1+(Nq_v+1)^a}+\dfrac{\sigma_{\zero}}{1+(Nq_v+1)^{-a}} \right). \notag
\end{align} 
From standard Markov process theory, for any function $f$ the process
\[ M^N_f(t) = f(Q^N (t),\sigma^N (t)) - f(Q^N(0),\sigma^N(0)) - \int_0^tL^N f(Q^N (s),\sigma^N (s))\d s \]
is a local martingale with increasing process
\[ \left\langle M^N_f \right\rangle(t)=\int_0^t\Gamma^N f(Q^N (s),\sigma^N (s))\d s. \]

For $N \geq 1$ we consider the homogenized generator $L^N_{\mathrm{h}}$ acting on functions $f: E^N \to \reels$ as
\begin{multline} \label{eq:Lh}
	L^N_{\mathrm{h}}f(q) = N^{a+1} \sum_{v \in V} \lambda^{N}_v \left( f\left( q + \frac{e^v}{N} \right) - f(q) \right)\\
	+ N^{a+1} \sum_{v \in V} \pi^{Nq}(v) \ind_{q_v>0} \left( f\left( q - \frac{e^v}{N} \right) - f(q) \right).
\end{multline}
This is the same generator as the generator $L^\sigma_{\mathrm{s}}$ of the (scaled) slow process given by~\eqref{eq:slow-L}, but where the instantaneous service rate $\sigma_v$ of node $v$ is replaced by its average value $\pi^q(v)$.

\subsubsection{Poisson equation}

For any function $g: V_0 \to \reels$ and any $q \in E^N$ we denote by $\phi^N_g(q, \, \cdot)$ the unique solution to the Poisson equation associated with the scaled fast generator $L^{N,q}_{\mathrm{f}}$ and the function $g$, i.e., $\phi^N_g(q, \, \cdot)$ is the unique solution with $\pi^{Nq}[\phi^N_g(q, \, \cdot)] = 0$ to the equation with unknown $\phi$
\begin{equation} \label{eq:Poisson}
	L^{N,q}_{\mathrm{f}} \phi = g - \pi^q[g].
\end{equation}
In the sequel, we will be particularly interested in $\phi^N_v(q, \, \cdot \,)$ solution to~\eqref{eq:Poisson} with $g(\sigma) = \sigma_v$ for $v \in V_0$, which therefore satisfies for any $q \in E^N$ and any $\sigma \in V_0$
\begin{equation} \label{eq:Poisson-i}
	L^{N,q}_{\mathrm{f}} \left( \phi^N_v(q, \, \cdot \,) \right)(\sigma) = \sigma_v - \pi^{Nq}(v).
\end{equation}

\subsubsection{Initial state, limiting ODE}
Recall that we fix throughout an initial state $q^0 \in I \setminus \{0\}$ and we assume that $Q^N(0) \to q^0$. Moreover, we consider $S = (S(t), t \geq 0)$ the solution to the ODE
\[ \dot x = \mu x^{-a}-\tmpsum(\gamma) \]
with initial condition $S(0) = \tmpsum(q^0)$.
We also consider $q = (q(t), t \geq 0)$ the $\mathbb{R}^V$-valued function with $\tmpsum \circ q = S$ and $q(t) \in I$ for all $t \geq 0$, i.e.,
\[ q_v(t) = \left( \frac{\lambda_v^\infty}{\mu} \right)^{1/a} S(t),\, t \geq 0,\, v \in V. \]
Note that for any choice of $\gamma\in \reels^V$, $S(t)$ is bounded away from zero, i.e., $\inf_{t \geq 0} S(t) > 0$. \textcolor{black}{If $\tmpsum(\gamma)=0$, $S$ has an explicit expression:}
\[ S(t) = \left( \mu (a+1) t + \tmpsum(q^0)^{a+1} \right)^{1/(a+1)}, \ t \geq 0. \]

\subsubsection{Localization, constants}

Most of the proof of Theorem~\ref{main} is carried out for a localized process $Q^N(t \wedge T^N)$ with $T^N$ the first time that $Q^N$ significantly departs from $q$. More precisely, in the rest of the paper we fix some finite time horizon $T > 0$ and we consider the following two constants:
\[ M = \min \left( 2 \sup_{[0,T]} S, \frac{2}{\inf_{[0,T]} S}, \frac{1}{2} \right) \ \text{ and } \ m = \frac{1}{M \mu^{1/a}} \min_v \left(\lambda^\infty_v\right)^{1/a}. \]
Here and in the sequel, we will treat as constants all numerical parameters that only depend on $a$, $n$, $T$, $\lambda^\infty$, $q^0$ and the sequence $(\lambda^N)$ as these are fixed throughout the entire paper. Moreover, we will use the letter $C$ to denote positive and finite constants, that only depend on $a$, $n$, $T$, $\lambda^\infty$, $q^0$ and $(\lambda^N)$, and whose precise value is irrelevant and that may change from line to line. Note in particular that the constants $C$ do not depend on $N$, so that if $0 \leq u_N \leq C v_N$ with $v_N \to 0$, then also $u_N \to 0$.

We then define
\[ T^N \coloneqq \inf \left \lbrace t>0 : \norm{Q^N(t)-q(t)}{1}>\dfrac{m}{2} \right \rbrace, \]
the set $U \subset \reels^V_+$
\[ U \coloneqq \left \lbrace q \in \reels_+^V: \frac{1}{M} < \tmpsum(q) < M \ \text{ and } \ \min_v q_v > m \right \rbrace, \]
its intersection $U^N$ with $E^N$
\[ U^N = U \cap E^N, \]
and the exit time of $Q^N$ from $U$ (or $U^N$):
\[ \tau^N \coloneqq \inf \left \lbrace t \geq 0 : Q^N(t) \notin U \right \rbrace. \]
Because jumps of $Q^N$ are of size $1/N$, at time $T^N$ we have
\[ \left \lVert Q^N(T^N) - q(T^N) \right \rVert \leq \frac{m}{2} + \frac{1}{N}. \]
The constants $m$ and $M$ have been chosen such that the following result holds. The proof is computational and omitted.

\begin{lem}\label{lemma:m-M}
	We have $T^N \leq \tau^N$. In particular, $Q^N(t \wedge T^N) \in U^N$ for all $t \geq 0$.
\end{lem}

\subsubsection{Distance to $I$}

In order to control the distance to the invariant manifold $I$ given by~\eqref{eq:I}, i.e., to control the state space collapse property, we will use the Kullback-Leibler divergence between $\lambda$ and $(\pi^q(v), v \in V)$ (note that the latter is not a probability measure). More precisely, for $q \in \R^V_+$ and $N \geq 1$ let
\[ d^N(q) = \sum_{v \in V} \lambda^\infty_v \log \left( \frac{\lambda^\infty_v}{\pi^{Nq}(v)} \right). \]
When $N \to \infty$ and $q \in U$ we have $\pi^{Nq}(v) \to \pi^q_\infty(v)$ where
\[ \pi^q_\infty(v) = \frac{q_v^a}{\norm{q}a^a}, \ v \in V. \]
We thus introduce
\[ d^\infty(q) = \sum_{v \in V} \lambda^\infty_v \log \left( \frac{\lambda^\infty_v}{\pi^q_\infty(v)} \right). \]
which therefore satisfies $d^N(q) \to d^\infty(q)$ as $N \to \infty$. The convergence is actually uniform in $q \in U$, as the next lemma states (the proof is omitted).

\begin{lem}\label{lem:d-1}
	As $N \to \infty$ we have
	\[ \sup_{q \in U^N} \left \lvert d^N(q) - d^\infty(q) \right \rvert \to 0. \]
\end{lem}

Note that $d^\infty(q) = 0$ if and only if $q \in I$, so $d^\infty$ can indeed be seen as a distance to $I$. \color{black} For $x \in I$ we have by definition $x_v = (\lambda^\infty / \mu)^{1/a} s(x)$. The distance $d^\infty$ to $I$ will actually also control the difference between $x_v$ and $(\lambda^\infty / \mu)^{1/a} s(x)$, and the next lemma will prove useful in the sequel.

\begin{lem} \label{lem:diff-I}
	For $x \in U$ we have
	\[ \left \lvert x_v - \left( \frac{\lambda^\infty_v}{\mu} \right)^{1/a} s(x) \right \rvert \leq C \left[ d^\infty(x) \right]^{1/2}, \ v \in V. \]
\end{lem}

\begin{proof}
	Because $y \in [m,M] \to y^{1/a}$ is Lipschitz, for $x \in [m,M]^V$ and $v \in V$ we have
	\[ \left \lvert x_v - \left( \lambda^\infty_v \right)^{1/a} \norm{x}a \right \rvert \leq C \left \lvert x^a_v - \lambda^\infty_v \norm{x}a^a \right \rvert \leq C \left \lvert \pi^x_\infty(v) - \lambda^\infty_v \right \rvert \]
	and so Pinsker's inequality gives
	\[ \left \lvert x_v - \left( \lambda^\infty_v \right)^{1/a} \norm{x}a \right \rvert \leq C \left[ d^\infty(x) \right]^{1/2}. \]
	Thus, since $s(x) = \sum_v x_v$ and $\mu^{1/a} = \sum_v (\lambda^\infty_v)^{1/a}$ we also have
	\[ \left \lvert s(x) - \mu^{1/a} \norm{x}a \right \rvert \leq \sum_{v \in V} \left \lvert x_v - \left( \lambda^\infty_v \right)^{1/a} \norm{x}a \right \rvert \leq C \left[ d^\infty(x) \right]^{1/2}. \]
	Finally, since
	\begin{multline*}
		\left \lvert x_v - \left( \frac{\lambda^\infty_v}{\mu} \right)^{1/a} s(x) \right \rvert \leq \left \lvert x_v - \left( \lambda^\infty_v \right)^{1/a} \norm{x}a \right \rvert\\
		+ \left \lvert \left( \lambda^\infty_v \right)^{1/a} \norm{x}a - \left( \frac{\lambda^\infty_v}{\mu} \right)^{1/a} s(x) \right \rvert,
	\end{multline*}
	we obtain the result.
\end{proof}

\color{black}

\subsection{Main steps}

The proof of Theorem~\ref{main} has three main steps which are proved in Sections~\ref{homo}--\ref{Smain}.

\subsubsection{First step: homogenization}

The first main step of the proof is the following averaging result: we give the main idea of its proof below, and defer the full proof to Sections~\ref{homo} and~\ref{sec:Poisson}. Recall that $C$ denotes a numerical constant allowed to depend on $a$, $n$, $T$, $\lambda^\infty$, $(\lambda^N)$ and $q^0$.

\begin{prop} \label{eqpoi}
	If $f: U \to \reels$ is continuously differentiable, then for any $v \in V$ we have
	\begin{multline*}
		\E \left[ \sup_{0 \leq t \leq T \wedge T^N} \left \lvert \int_0^t \left( \sigma^N_v(s)-\pi^{NQ^N(s)}(v) \right) f \left( Q^N(s) \right) \d s \right \rvert \right]\\
		\leq C \norm{f}{\infty, U}\dfrac{(\log N)^{3/2}}{N^{1/2}} + C \max_v \norm{\partial_v f}{\infty, U} \dfrac{(\log N)^{3/2}}{N^{1-a}}.
	\end{multline*}
\end{prop}
The proof of this result has two steps: first, provide a bound in terms of solutions to the Poisson equation~\eqref{eq:Poisson} and then controlling these solutions. These two steps are performed in Sections~\ref{homo} and~\ref{sec:Poisson}, respectively \color{black} and, as far as we know, the bounds that we derive there are new\color{black}. To see how the Poisson equation arises, let us proceed with the following preliminary computation. We get from~\eqref{eq:Poisson-i}
\[ \sigma^N_v(s) - \pi^{NQ^N(s)}(v) = L^{N, Q^N(s)}_{\mathrm{f}} \left( \phi^N_v(Q^N(s),\, \cdot \,) \right) \left( \sigma^N(s) \right). \]
Since $f$ does not depend on $\sigma$, this makes it possible to rewrite
\begin{multline*}
	\int_0^t \left( \sigma^N_v(s)-\pi^{NQ^N(s)}(v) \right) f \left( Q^N(s) \right)\d s\\
	= \int_0^t L^{N,Q^N(s)}_{\mathrm{f}} \left( V^N_v(Q^N(s),\, \cdot \,) \right) \left( \sigma^N(s) \right) \d s
\end{multline*}
with $V^N_v(q, \sigma) = \phi^N_v(q, \sigma) f(q)$. Making use of the martingale decomposition, we finally rewrite this as
\begin{multline} \label{eq:poisson}
	\int_0^t \left( \sigma^N_v(s)-\pi^{NQ^N(s)}(v) \right) f \left( Q^N(s) \right)\d s\\
	= V^N_v \left (Q^N(t), \sigma^N(t) \right) - V^N_v \left (Q^N(0), \sigma^N(0) \right)\\
	- \int_0^t L^{N, \sigma^N(s)}_{\mathrm{s}} \left( V^N_v(\, \cdot \,, \sigma^N(s)) \right) \left( Q^N(s) \right) \d s - M^N_{V^N_v}(t).
\end{multline}
This expression will be the basis for the proof of Proposition~\ref{eqpoi}.

\subsubsection{Second step: state space collapse}

Using the averaging result of Proposition~\ref{eqpoi}, the next step is to prove the following state space collapse result.

\begin{prop} \label{prop:SSC}
	As $N \to \infty$ we have
	\[ \E \left[ \sup_{0 \leq t \leq T \wedge T^N} d^\infty \left( Q^N(t) \right) \right] \to 0. \]
\end{prop}

The proof proceeds by controlling the action of the homogenized generator $L^N$ on $d^N$ and then use this result to control $d^\infty \circ Q^N$ thanks to the averaging result of Proposition~\ref{eqpoi}.

\subsubsection{Third step: full proof}

The third step of the proof consists in showing that $Q^N(\, \cdot \wedge T^N) \Rightarrow q$. The proof proceeds in two steps: first we establish the convergence of the one-dimensional total queue length process $\tmpsum \circ Q^N(\, \cdot \wedge T^N) \Rightarrow \tmpsum \circ q = S$ by using Gronwall's lemma. Together with the state space collapse property of Proposition~\ref{prop:SSC}, this gives the convergence of the entire $n$-dimensional process $Q^N(\, \cdot \wedge T^N)$ stopped at time $T^N$.

We finally conclude the proof: because the limiting process $q$ does not exit the set $U$ by time $T$, we prove that with high probability $Q^N$ also stays in $U$ by time $T$: this implies in particular that $\P(T^N \geq T) \to 1$ which makes it possible to transfer the convergence result from the stopped process $Q^N(\, \cdot \wedge T^N)$ to the unstopped one $Q^N$.

\section{Control of homogenization in terms of solutions to the Poisson equation} \label{homo}

This section provides a first step toward the proof of Proposition~\ref{eqpoi}. We first derive a bound in terms of the following constants:
\[ \Omega_N\coloneqq \sup_{q\in U^N,\, \norm{g}{\infty}\leq 1} \norm{\phi^N_g(q,\, \cdot \,)}{\infty}, \]
\[ B_N\coloneqq\sup_{q\in U^N,\, \norm{g}{\infty}\leq 1}\max_{i\in V, \sigma \in V_0}\left\lvert\Delta^N_{\pm, i}\phi^N_g(q,\sigma)\right\rvert \]
and
\[ \Theta_N = N^{a+1} B_N + N^{1/2} \Omega_N + N^{(a+1)/2} \Omega_N^{3/2} + N^{a+1} \Omega_N B_N^{1/2} \]

\begin{lem}\label{lem:Poisson-1}
	For any $v \in V$ we have
	\begin{multline*}
		\E \left[ \sup_{0 \leq t \leq T \wedge T^N} \left\lvert \int_0^t \left( \sigma^N_v(s)-\pi^{NQ^N(s)}(v) \right) f \left( Q^N(s) \right)\d s \right\rvert \right]\\
		\leq C \lVert f \rVert_{\infty, U} \Theta_N + C \max_w \lVert \partial_w f \rVert_{\infty, U} \left( N^{(a+1)/2} B_N + N^a \Omega_N \right).
	\end{multline*}
\end{lem}

Recall from the arguments preceding~\eqref{eq:poisson} that $V^N_v(q, \sigma) = f(q) \phi^N_v(q, \sigma)$: thus for every $v$ we have for $q \in U^N$
\begin{equation} \label{eq:bound-V}
	\left \lvert V^N_v (q, \sigma) \right \rvert \leq \lVert f \rVert_{\infty, U} \Omega_N
\end{equation}
and
\begin{equation} \label{eq:bound-Delta-V}
	\left \lvert \Delta^N_{\pm, w} V^N_v(q, \sigma) \right \rvert \leq \max_{w} \lVert \partial_{w} f \rVert_{\infty, U} \frac{\Omega_N}{N} + \lVert f \rVert_{\infty, U} B_N.
\end{equation}
We start with two preliminary lemmas.

\begin{lem}\label{lem:0}
	We have
	\begin{align*}
		\E \left[ \int_0^{T \wedge T^N} \sigma^N_{\zero}(s) \d s \right] \leq C N^{-a} + C \Omega_N + C N^{a+1} B_N.
	\end{align*}
\end{lem}

\begin{proof}
	Note that
	\[ \pi^{N Q^N(s)}(\zero) = \frac{1}{1 + \sum_{w \in V} (N Q^N_{w}(s) + 1)^a} \]
	and so since $Q^N_w(s) \geq m - 1/N$ for $t \leq T^N$, we have $\pi^{N Q^N(s)}(0) \leq C N^{-a}$ for $s \leq T^N$ and so
	\[ \E \left[ \int_0^{T \wedge T^N} \sigma^N_{\zero}(s) \d s \right] \leq C N^{-a} + \E \left[ \int_0^{T \wedge T^N} \left( \sigma^N_{\zero}(s) - \pi^{N Q^N(s)}(\zero) \right) \d s \right]. \]
	Starting from~\eqref{eq:poisson} with $f=1$ and taking the mean, we obtain
	\begin{multline*}
		\E \left[ \int_0^{T \wedge T^N} \left( \sigma^N_{\zero}(s) - \pi^{N Q^N(s)}(\zero) \right) \d s \right]\\
		= \E \left[ \phi^N_\zero \left( Q^N(T \wedge T^N), \sigma^N(T \wedge T^N) \right) \right] - \phi^N_\zero \left( Q^N(0), \sigma^N(0) \right)\\
		- \E \left[ \int_0^{T \wedge T^N} L^{N, \sigma^N(s)}_{\mathrm{s}} \left( \phi^N_\zero(\, \cdot \,, \sigma^N(s)) \right) \left( Q^N(s) \right) \d s \right].
	\end{multline*}
	By definition of $L^{N, \sigma}_{\mathrm{s}}$ we have
	\begin{multline*}
		L^{N, \sigma^N(s)}_{\mathrm{s}} \left( \phi^N_\zero(\, \cdot \,, \sigma^N(s)) \right) \left( Q^N(s) \right) = N^{a+1} \sum_{v \in V} \lambda^{N}_v \Delta^N_{+, v} \phi^N_\zero(Q^N(s), \sigma^N(s))\\
		+ N^{a+1} \sum_{v \in V} \sigma^N_v(s) \ind_{Q^N_v(s) > 0} \Delta^N_{-, v} \phi^N_\zero(Q^N(s), \sigma^N(s)).
	\end{multline*}
	The result thus follows directly from the definitions of $\Omega_N$ and $B_N$ since $Q^N(t \wedge T^N) \in U$ according to Lemma~\ref{lemma:m-M}.
\end{proof}

\begin{lem} \label{lem:M}
	We have
	\begin{multline*}
		\E \left[ \sup_{0 \leq t \leq T\wedge T^N} \left\lvert M^N_{V^N_v}(t)\right\rvert \right] \leq C \max_{w} \lVert \partial_{w} f \rVert_{\infty, U} N^{(a+1)/2} B_N + C \lVert f \rVert_{\infty, U} \Theta_N.
	\end{multline*}
\end{lem}

\begin{proof}
	\textcolor{black}{By Doob's inequality and Itô's isometry} we have
	\begin{align*}
		\cesp{}{\sup_{t\leq T\wedge T^N} \left( M^N_{V^N_v}(t) \right)^2} & \leq 4 \cesp{}{ \left( M^N_{V^N_v}(T\wedge T^N) \right)^2}\\
		& = 4\cesp{}{ \langle M^N_{V^N_v} \rangle(T\wedge T^N)}\\
		& = 4\cesp{}{\int_0^{T\wedge T^N}\Gamma^NV^N_v(Q^N(s),\sigma^N(s))\d s}.
	\end{align*}
 According to~\eqref{eq:Gamma}, we have
	\begin{align*}
		\Gamma^N & V^N_v (q,\sigma) = N^{a+1} \sum_{w \in V} \lambda^{N}_{w} \left( \Delta^N_{+,w} V^N_v (q, \sigma)\right)^2\\
		& + N^{a+1} \sum_{w \in V} \sigma_{w} \ind{q_{w} > 0} \left( \Delta^N_{-,{w}} V^N_v (q, \sigma)\right)^2\\
		& + N^{a+1} \sum_{w \in V} \left( V^N_v(q,0)- V^N_v(q,{w}) \right)^2 \dfrac{\sigma_{w}}{1+( \textcolor{blue}{N} q_{w}+1)^a} \notag\\
		& + N^{a+1} \sum_{{w \in V}} \left( V^N_v(q,0)- V^N_v(q,{w}) \right)^2 \dfrac{\sigma_{\zero}}{1+(\textcolor{blue}{N}q_{w}+1)^{-a}}. \notag
	\end{align*}
	We integrate this quantity over the trajectory $(Q^N, \sigma^N)$ for $t \leq T \wedge T^N$: along this trajectory we bound the terms $\sigma^N_{w}(s)$ and $1/(1+(N Q^N_{w}(s)+ 1)^{-a})$ by one, the terms $1/(1+(NQ^N_{w}(s)+ 1)^a)$ by $C N^{-a}$ (because $Q^N_{w}(s) \geq m -1/N$ for $t \leq T^N$) and we use~\eqref{eq:bound-V} and~\eqref{eq:bound-Delta-V} to obtain
	\begin{align*}
		\E & \left[ \int_0^{T \wedge T^N} \Gamma^N V^N_v (Q^N(s), \sigma^N(s)) \d s \right]\\
		& \hspace{30mm} \leq C N^{a+1} \left( \max_{w} \lVert \partial_{w} f \rVert_{\infty, U} \frac{\Omega_N}{N} + B_N \lVert f \rVert_{\infty, U} \right)^2\\
		& \hspace{35mm} + C \lVert f \rVert_{\infty, U}^2 N \Omega_N^2 \notag\\
		& \hspace{35mm} + C \lVert f \rVert_{\infty, U}^2 N^{a+1} \Omega_N^2 \E \left[ \int_0^{T \wedge T^N} \sigma^N_{\zero}(s) \d s \right].
	\end{align*}
	Using $(x+y)^2 \leq 2x^2 + 2y^2$ and Lemma~\ref{lem:0}, we therefore obtain
	\begin{multline*}
		\E \left[ \sup_{0 \leq t \leq T \wedge T^N} M^N_{V^N_v}(t)^2 \right] \leq C \max_{w} \lVert \partial_{w} f \rVert_{\infty, U}^2 N^{a+1} B_N^2\\
		+ C \lVert f \rVert_{\infty, U}^2 \left( N^{a+1} B_N^2 + N \Omega_N^2 + N^{a+1} \Omega_N^3 + N^{2a+2} \Omega_N^2 B_N \right).
	\end{multline*}
	The result then follows by Cauchy-Schwarz and sub-linearity of the square root, and also because
	\[ N^{(a+1)/2} B_N + N^{1/2} \Omega_N + N^{(a+1)/2} \Omega_N^{3/2} + N^{a+1} \Omega_N B_N^{1/2} \leq C \Theta_N. \]
\end{proof}

\begin{proof} [Proof of Lemma~\ref{lem:Poisson-1}]
	Starting from~\eqref{eq:poisson}, we obtain
	\begin{align*}
		\sup_{0 \leq t \leq T \wedge T^N} & \left\lvert \int_0^t \left( \sigma^N_v(s)-\pi^{NQ^N(s)}(v) \right) f \left( Q^N(s) \right)\d s \right\rvert\\
		& \leq \left \lvert V^N_v(Q^N(0),\sigma^N(0)) \right \rvert + \sup_{0 \leq t \leq T \wedge T^N} \left \lvert V^N_v(Q^N(t),\sigma^N(t)) \right \rvert \notag\\
		& \hspace{5mm} + \sup_{0 \leq t\leq T\wedge T^N} \left\lvert \int_0^t L^{N, \sigma^N(s)}_{\mathrm{s}} \left( V^N_v \left(\, \cdot \, ,\sigma^N(s) \right) \right) (Q^N(s))\d s \right\rvert \notag\\
		& \hspace{5mm} + \sup_{0 \leq t \leq T\wedge T^N} \left\lvert M^N_{V^N_v}(t)\right\rvert. \notag
	\end{align*}

	As $Q^N(t) \in U$ for $t \leq T^N$ by Lemma~\ref{lemma:m-M}, similar arguments as in the proof of Lemmas~\ref{lem:0} and~\ref{lem:M} give a control on the three first terms in the right-hand side of the previous display, namely
	\[ \left \lvert V^N_v(Q^N(0),\sigma^N(0)) \right \rvert + \sup_{0 \leq t \leq T \wedge T^N} \left \lvert V^N_v(Q^N(t),\sigma^N(t)) \right \rvert \leq C \lVert f \rVert_{\infty, U} \Omega_N \]
	and
	\begin{multline*}
		\sup_{0 \leq t\leq T\wedge T^N} \left\lvert \int_0^t L^{N, \sigma^N(s)}_{\mathrm{s}} \left( V^N_v \left(\, \cdot \, ,\sigma^N(s) \right) \right) (Q^N(s))\d s \right\rvert\\
		\leq C N^a \Omega_N \max_{w} \lVert \partial_{w} f \rVert_{\infty, U} + C \lVert f \rVert_{\infty, U} N^{a+1} B_N.
	\end{multline*}
	Combining these bounds with the bound of Lemma~\ref{lem:M} gives the result.
\end{proof}

\section{Control of solutions to the Poisson equation} \label{sec:Poisson}

In the previous section we have established a bound on some averaging property in terms of the constants $\Omega_N$ and $B_N$. The goal of this section is to prove the following result which provides a bound on these constants.

\begin{lem}\label{ordre}
	We have the following two bounds:
	\[ \Omega_N \leq C \frac{(\log N)^{3/2}}{N} \ \text{ and } \ B_N \leq C \frac{(\log N)^3}{N^2}. \]
\end{lem}

We will prove this in a series of lemmas. It is more convenient to focus on unscaled quantities. For $q \in \N^V$ let $\alpha^q$ and $\ell^q$ be the log-Sobolev constant and spectral gap associated with $L^q_{\mathrm{f}}$, respectively, and $\phi_g(q, \, \cdot \,)$ the solution to the Poisson equation $L^q_{\mathrm{f}} \varphi = g - \pi^q[\varphi]$.

\begin{lem} \label{lemma:Omega-B}
	For $q \in \N^V$ and $v\in V$ let
	\[ \Omega(q) = \frac{(\log (1/\pi^q(\zero)))^{1/2} \log(1/\pi^q(\zero)-1)}{\ell^q (1-2\pi^q(\zero))} \]
	and
	\[ B_v(q) = \frac{\Omega(q)}{q^{1-a}_v} \left( \frac{4 \Omega(q)}{q^{2a}_v} + \pi^q(\zero) \right). \]
	Then
	\begin{equation} \label{eq:bound-phi}
		\lVert \phi_g(q, \, \cdot \,) \rVert_\infty \leq \lVert g \rVert_\infty \Omega(q)
	\end{equation}
	and
	\begin{equation} \label{eq:bound-Delta-phi}
		\lVert \phi_g(q \pm e^v, \, \cdot \,) - \phi_g(q, \, \cdot \,) \rVert_\infty \leq \lVert g \rVert_\infty B_v(q).
	\end{equation}
\end{lem}

\begin{proof}
	Let $m^q_{\sigma, t}$ denote the law at time $t$ of the Markov process starting at $\sigma$ with generator $L^q_{\mathrm{f}}$: then it is well-known that $\phi_g(q, \, \cdot, \,)$ is given by
	\[ \phi_g(q, \sigma) = - \int_0^\infty \left( m^q_{\sigma, t}[g] - \pi^q[g] \right) \d t. \]
	This gives
	\[ \norm{\phi_g(q, \, \cdot)}{\infty} \leq 2\norm{g}{\infty}\int_0^{+\infty}\norm{m^q_{\sigma, t}-\pi^q}{\textrm{TV}}\d t \]
	with $\lVert \, \cdot \, \rVert_{\textrm{TV}}$ the total variation distance and then
	\[ \norm{\phi_g(q, \, \cdot)}{\infty} \leq 2\norm{g}{\infty}\int_0^{+\infty} \left( \dfrac{1}{2} m^q_{\sigma, t} \left[\varphi^q_{\sigma, t} \right] \right)^{1/2} \d t \]
	where $\varphi^q_{\sigma, t} = \log(m^q_{\sigma, t} / \pi^q)$, by Pinsker's inequality. As $\min_\sigma \pi^q(\sigma) = \pi^q(0)$, Theorem $3.6$ in~\cite{Dia96} gives
	\[ m^q_{\sigma, t} \left[\varphi^q_{\sigma, t} \right] \leq \log(1/\pi^q(0)) e^{-4 \alpha^q t} \]
	while Corollary $2.2.10$ in~\cite{Sal97} gives
	\[ \alpha^q \geq \frac{1-2\pi^q(\zero)}{\log((1-\pi^q(\zero))/\pi^q(\zero))} \ell^q. \]
	Gathering the three previous bounds gives the desired bound~\eqref{eq:bound-phi} on $\lVert \phi_g(q, \, \cdot \,) \lVert_\infty$. We now prove~\eqref{eq:bound-Delta-phi}. Fix temporarily $v \in V$, $q \in \N^V$ with $q_v>0$ and let $\Phi = \phi_g(q-e^v,\, \cdot \,) - \phi_g(q,\, \cdot \,)$ and $G = L^q_{\mathrm{f}}(\Phi)$. Since $\pi^q(G)=0$, the first bound~\eqref{eq:bound-phi} thus implies
	\[ \norm{\phi_g( q-e^v,\, \cdot \,) - \phi_g(q,\, \cdot \,)}{\infty} \leq \Omega(q)\norm{G}{\infty}. \]
	Since by definition of $\phi_g$ we have $L^q_{\mathrm{f}}(\phi_g( q,\, \cdot \,))(\sigma)=g(\sigma) - \pi^q[g]$ we obtain
	\[ G(\sigma) = - \left( L^{q-e^v}_{\mathrm{f}} - L^q_{\mathrm{f}} \right) \left( \phi_g( q-e^v,\, \cdot \,) \right) (\sigma) - \sum_{\rho \in V_0}(\pi^{ q-e^v}(\rho)-\pi^q(\rho))g(\rho) \]
	and so
	\[\norm{G}{\infty} \leq \norm{\left( L^{q-e^v}_{\mathrm{f}} - L^q_{\mathrm{f}} \right) \left( \phi_g( q-e^v,\, \cdot \,) \right)}{\infty} + \norm{g}{\infty}\sum _{\rho \in V_0}\left\lvert \pi^{ q- e^v}(\rho)-\pi^q(\rho)\right\rvert. \]
	For any function $h: V_0 \to \R$ we have according to~\eqref{eq:fast-L}
	\begin{multline*}
		\left( L^{q-e^v}_{\mathrm{f}} - L^q_{\mathrm{f}} \right) (h)(\sigma) = \sigma_v \left( \Psi_-(q_v-1) - \Psi_-(q_v) \right) \left( h(\sigma - e^v) - h(\sigma) \right)\\
		+ \sigma_0 \left( \Psi_+(q_v-1) - \Psi_+(q_v) \right) \left( h(\sigma + e^v) - h(\sigma) \right)
	\end{multline*}
	and so since $\Psi_+ + \Psi_- = 1$, this gives
	\[ \left \lVert \left( L^{q-e^v}_{\mathrm{f}} - L^q_{\mathrm{f}} \right) (h) \right \rVert_\infty \leq 4 \left \lVert h \right \rVert_\infty \left \lvert \Psi_-(q_v-1) - \Psi_-(q_v) \right \rvert. \]
	Therefore, using again the bound~\eqref{eq:bound-phi} gives
	\[ \norm{\left( L^{q-e^v}_{\mathrm{f}} - L^q_{\mathrm{f}} \right) \left( \phi_g( q-e^v,\, \cdot \,) \right)}{\infty} \leq 4 \Omega(q) \norm{g}{\infty} \int_0^{1}\left\lvert \Psi_d'( q_v-u)\right\rvert \d u. \]
	Direct calculation yields
	\[ \Psi_-'(q_v) = -\dfrac{a}{(q_v+1)^{1-a}(1+(q_v+1)^a)^2} \]
	and so $\left \lvert \Psi_-'(q_v-u) \right \rvert \leq \frac{q_v^{a-1}}{(1+q_v^a)^2}$ as long as $u\leq 1$. Moreover, for any $v \in V$ and $w \in V_0$ with $w \neq v$, one can check that
	\[ \left\lvert \partial_v \pi^q(w) \right\rvert = a(q_v+1)^{a-1}\pi^{q}(w)\pi^q(\zero) \]
	and
	\[ \left\lvert \partial_v \pi^q(v) \right\rvert = a(q_v+1)^{a-1}\pi^q(\zero)(1-\pi^q(v)) \]
	so that in any case, $\left \lvert \partial_v \pi^q(\sigma) \right \rvert \leq a\pi^q(\zero)(q_v+1)^{a-1}$. Gathering the previous bounds gives the result.
\end{proof}

We now prove a lower bound on the spectral gap of $L^q_{\mathrm{f}}$. Related bounds were for instance proved in~\cite{shashin} using Cheeger's inequality in a more general setting. However, this method would only lead to $\ell^q \geq C \lVert q+1 \rVert^{-2a}_\infty$ which is not sharp enough in our case.

\begin{lem} \label{lemma:lambda}
	For any $q \in \mathbb{N}^V$ we have
	\[ \ell^q \geq \frac{C}{\lVert q + 1 \rVert_\infty^a}. \]
\end{lem}

\begin{proof}
	Let $(\eta(t), t \geq 0)$ be a Markov process with generator $L^q_{\mathrm{f}}$ and $\mathbb{P}^q_\sigma$ its law started from $\sigma \in V_0$. Let as in the previous proof $m^q_{\sigma, t}$ denote the law of $\eta(t)$ under $\mathbb{P}^q_\sigma$ and define the random times
	\[T_\mix^q = \inf \left \{t \geq 0: \max_{\sigma \in V} \lVert m^q_{\sigma, t} - \pi^q \rVert_{\textrm{TV}} < \frac{1}{2e} \right \} \]
	and
	\[ T_\hit^q = \max_{\sigma \in V, A \subset V_0} \pi^q(A) \E^q_{\sigma^0}(T_A) \]
	with $T_A$ the hitting time of $A$ for $\eta$:
	\[ T_A = \inf \left\{ t \geq 0: \eta(t) \in A \right\}, \ A \subset V_0. \]
	Recall that $\ell^q$ is the spectral gap of $L^q_{\mathrm{f}}$. It is proved in~\cite{lev} that $T^q_\mix \geq 1/\ell^q - 1$ (the proof for discrete time extends to continuous time) and in~\cite{ald} that $T^q_\mix \leq c_0 T^q_\hit$ for some universal constant $c_0$. Combining those two results, we get that
	\[ \ell^q \geq \dfrac{1}{c_0T^q_\hit+1} \]
	and so in order to prove the desired bound, we only need to prove that $T^q_\hit \leq C \lVert q+1 \rVert_\infty^a$. Since
	\[ T^q_\hit\leq \max_{\sigma^0 \in V_0}\sum_{\sigma \in V_0} \E^q_{\sigma^0} (T_{\sigma}) \]
	this actually reduces to proving that
	\begin{equation} \label{eq:goal-Thit}
		\E^q_\sigma(T_\zero) \leq C \lVert q + 1 \rVert^a_\infty \ \text{ and } \ \E^q_\zero(T_\sigma) \leq C \lVert q + 1 \rVert^a_\infty
	\end{equation}
	for any $\sigma \in V$. Indeed, for $\sigma^0 \neq \sigma \in V$ the process $\eta$ needs to pass through~$\zero$ to go from $\sigma^0$ to $\sigma$ and so the strong Markov property gives
	\[ \E^q_{\sigma^0}(T_\sigma) = \E^q_{\sigma^0}(T_\zero) + \E^q_\zero(T_\sigma). \]
	So let us prove~\eqref{eq:goal-Thit}. The bound on $\E^q_\sigma(T_\zero)$ is obvious since by definition $T_\zero$ under $\E^q_\sigma$ is an exponential random variable with parameter $\Psi_-(q_\sigma)$ so that
	\[ \E^q_\sigma(T_\zero) = \frac{1}{\Psi_-(q_\sigma)} = 1 + \left( q_v + 1 \right)^a \leq C \lVert q + 1 \rVert^a_\infty. \]
	Let us now prove that $\E^q_\zero(T_\sigma) \leq C \lVert q + 1 \rVert^a_\infty$. Under $\P^q_\zero$, decompose the trajectory $(\eta(t), 0 \leq t \leq T_\sigma)$ into cycles away from $\zero$: in the $a$-th cycle, $\eta$ stays in $\zero$ for a duration $X_a$, then moves to some $i \in V$ where it stays for a duration $Y_a$ and then comes back to $\zero$. If $A \in \{1, \ldots, \}$ denotes the first cycle where $\eta$ visits $\sigma$, we can thus write
	\[ T_\sigma = \sum_{a=1}^{A-1} (X_a + Y_a) + X_A. \]
	Each time $\eta$ leaves $\zero$, it goes to $i \in V$ with probability
	\[ p^q_v = \frac{\Psi_+(q_v)}{\sum_{w \in V} \Psi_+(q_v)} = \left( \sum_{w \in V} \frac{1 + (q_v+1)^{-a}}{1 + (q_{w}+1)^{-a}} \right)^{-1} \geq \frac{1}{2n}. \]
	In particular, with a suitable coupling we can write
	\[ T_\sigma \leq \sum_{a=1}^G (X_a + Y_a) \]
	with $G$ a geometric random variable with parameter $1/(2n)$ independent from the $X_a$ and $Y_a$'s. Since the $(X_a, a \geq 1)$ and $(Y_a, a \geq 1)$ are two independent sequences of i.i.d.\ random variables, it follows that
	\[ \E^q_\zero (T_\sigma) \leq \E(G) \left[ \E^q_\zero (X_1) + \E^q_\zero (Y_1) \right]. \]
	By definition, $X_1$ under $\P^q_\zero$ is an exponential random variable with parameter
	\[ \sum_{v \in V} \Psi_+(q_v) = \sum_{v \in V} \frac{1}{1 + (q_v+1)^{-a}} \geq \frac{n}{2} \]
	so that $\E^q_\zero(X_1) \leq 2/n$. Moreover, $Y_1$ under $\P^q_\zero$ is distributed as $T_\zero$ under $\P^q_\Sigma$ with $\P(\Sigma = v) = p^q_v$ for $v \in V$, so that
	\[ \E^q_\zero(Y_1) = \sum_{v \in V} p^q_v \E^q_v (T_\zero) \leq C \lVert q+1 \rVert^a_\infty \]
	using $\E^q_v(T_\zero) \leq C \lVert q+1 \rVert^a_\infty$. Gathering the previous bounds yields the desired result.
\end{proof}

Thanks to Lemmas~\ref{lemma:Omega-B} and~\ref{lemma:lambda} we now provide a proof of Lemma~\ref{ordre}.

\begin{proof} [Proof of Lemma~\ref{ordre}]
	Let $g: V_0 \to \R$ and $q \in E^N$ given. Recall that $\phi^N_g(q, \, \cdot ,)$ and $\phi_g(Nq, \, \cdot ,)$ are such that
	\[ L^{N,q}_{\mathrm{f}} \left( \phi^N_g(q, \, \cdot ,) \right) = g - \pi^{Nq}[g] = L^{Nq}_{\mathrm{f}} \left( \phi_g(Nq, \, \cdot ,) \right) \]
	and since $L^{N,q}_{\mathrm{f}} = N^{a+1} L^{Nq}_{\mathrm{f}}$, this gives $\phi^N_g(q, \sigma) = N^{-(a+1)} \phi_g(Nq, \sigma)$ by uniqueness. According to~\eqref{eq:bound-phi} and~\eqref{eq:bound-Delta-phi} this gives
	\[ \left \lVert \phi^N_g(q, \,\cdot\,) \right \rVert_\infty \leq \left \lVert g \right \rVert_\infty \frac{\Omega(Nq)}{N^{a+1}} \ \text{ and } \ \left \lVert \Delta^N_{\pm, v} \phi^N_g(q, \, \cdot \,) \right \rVert_\infty \leq \lVert g \rVert_\infty \frac{B_v(Nq)}{N^{a+1}} \]
	and so in order to prove the result, we only have to prove that
	\[ \Omega(Nq) \leq C N^{a+1} \frac{(\log N)^{3/2}}{N} \ \text{ and } \ B_v(Nq) \leq C N^{a+1} \frac{(\log N)^3}{N^2} \]
	for $q \in U^N$. To do so, note that for $q \in U^N$ we have
	\[ n(Nm)^a \leq \frac{1}{\pi^{Nq}(\zero)} \leq 1+n(NM+1)^a \]
	and so our convention makes it possible to write $C N^{-a} \leq \pi^{Nq}(\zero) \leq C N^{-a}$. Since $\ell^{Nq} \geq C N^{-a}$ by Lemma~\ref{lemma:lambda}, for $q \in U^N$ we obtain the desired bounds on $\Omega(Nq)$ and $B_v(Nq)$.
\end{proof}

\begin{proof}[Proof of Proposition~\ref{eqpoi}]
	The proof of Proposition~\ref{eqpoi} now follows readily from Lemmas~\ref{lem:Poisson-1} and~\ref{ordre} since for $a < 1/2$ one readily checks that $\Theta_N \leq C (\log N)^{3/2}/N^{1/2}$ and
	\[ N^{(a+1)/2} B_N + N^a \Omega_N \leq C \frac{(\log N)^{3/2}}{N^{1-a}}. \]
\end{proof}

\section{State space collapse} \label{Sssc}

In this section we prove Proposition~\ref{prop:SSC} through a series of lemmas. \color{black} In view of Lemma~\ref{lem:d-1}, it is enough to prove the result with $d^N$ instead of $d^\infty$, i.e., to prove that
\[ \E \left[ \sup_{0 \leq t \leq T \wedge T^N} d^N(Q^N(t)) \right] \to 0. \]
Starting from the semimartingale decomposition of $d^N \circ Q^N$ and then adding and subtracting $L^N_\mathrm{h}$, we obtain
\begin{multline*}
	d^N(Q^N(t)) = d^N(Q^N(0)) + \int_0^t L^N_{\mathrm{h}} d^N(Q^N(s))\d s\\
	+ \int_0^t (L^N - L^N_{\mathrm{h}}) d^N (Q^N(s), \sigma^N(s))\d s + M^N_{d^N}(t)
\end{multline*}
where, in order to give sense to $L^N d^N$ we consider $d^N(q, \sigma) = d^N(q)$. Taking the supremum and the expectation and using that $Q^N(t) \in U^N$ for all $t \leq T^N$, this leads to
\begin{equation} \label{eq:bound-d^N}
	\E \left[ \sup_{0 \leq t \leq T \wedge T^N} d^N(Q^N(t)) \right] \leq d^N(Q^N(0)) + T \sup_{q \in U^N} L^N_\mathrm{h} d^N(q) + I + II
\end{equation}
with
\[ I = \E \left[ \sup_{0 \leq t \leq T \wedge T^N} \left \lvert \int_0^{t\wedge T^N} (L^N - L^N_{\mathrm{h}}) d^N (Q^N(s), \sigma^N(s))\d s \right \rvert \right] \]
and
\[ II = \E \left[ \sup_{0 \leq t \leq T \wedge T^N} \left \lvert M^N_{d^N}(t) \right \rvert \right]. \]
The first term $d^N(Q^N(0))$ in the right-hand side of~\eqref{eq:bound-d^N} vanishes because $Q^N(0) \to q^0 \in I$ (and because of Lemma~\ref{lem:d-1}). The next three lemmas show that $\sup_{q \in U^N} L^N_\mathrm{h} d^N(q) \to 0$ and that the terms $I$ and $II$ also vanish.

\begin{lem}\label{prop2}
	For $N$ large enough, we have for any $q \in E^N \cap U$
	\[ L^N_{\mathrm{h}}d^N(q) \leq C N^{-a} + C N^{-(1-a)}. \]
\end{lem} 
%

\begin{proof}
	Let $q \in U^N$ and for each $v \in V$, let $\zeta^v_{\pm} = (\zeta^v_{\pm, w}, w \in V)$ such that
	\[ d^N \left( q \pm \frac{e^v}{N} \right) = d^N(q) \pm \frac{1}{N} \partial_v d^N(q) + \frac{1}{2N^2} \partial^2_{v,v} d^N(\zeta^v_\pm). \]
	Note that $\zeta^v_{\pm, w} = q_w$ if $w \neq v$ and $\lvert \zeta^v_{\pm, v} - q_v \rvert \leq 1/N$. Then, recalling that $\Delta^N_{\pm, v} d^N(q) = d^N \left( q \pm e^v / N \right) - d^N(q)$, we obtain
	\begin{align*}
		L^N_{\mathrm{h}}d^N(q) & = N^{a+1} \sum_{v\in V} \lambda^{N}_v \Delta^N_{+,v} d^N(q) + N^{a+1} \sum_{v\in V} \pi^{Nq}(v) \Delta^N_{-,v} d^N(q) \\
	& = N^{a+1} \sum_{v\in V} \left(\lambda^\infty_v- N^{-a}\gamma_v\right) \left( \frac{1}{N}\partial_v d^N(q) + \frac{1}{2N^2} \partial^2_{v,v} d^N (\zeta^v_+) \right)\\
	& \hspace{5mm} + N^{a+1} \sum_{v\in V} \pi^{Nq}(v) \left( -\frac{1}{N}\partial_v d^N(q) + \frac{1}{2N^2} \partial^2_{v,v} d^N (\zeta^v_-) \right)\\
	& = A + B
	\end{align*}
	with
	\[ A = N^a \sum_{v\in V} \left( \lambda^\infty_v - \pi^{Nq}(v) \right) \partial_v d^N(q) - \sum_{v\in V} \gamma_v \partial_v d^N(q) \]
	and
	\[ B = \frac{1}{2 N^{1-a}} \sum_{v\in V} \left( \left( \lambda^\infty_v - N^{-a} \gamma_v \right) \partial^2_{v,v} d^N(\zeta^v_+) + \pi^{Nq}(v) \partial^2_{v,v} d^N (\zeta^v_-) \right). \]
	We now show that $A \leq C N^{-a}$ and $B \leq C N^{-(1-a)}$, which will give the result. Let us start with controlling $A$. Let in the sequel $\delta^N_v(q) = \lambda^\infty_v - \pi^{Nq}(v)$. Then it may be checked through elementary algebra that
	\[ \partial_v d^N(q) = - \frac{a \delta^N_v(q)}{q_v + \frac{1}{N}} \]
	which leads to the relation
	\[ A = - a N^a \sum_{v \in V} \frac{\delta^N_v(q)^2}{q_v + \frac{1}{N}} - a \sum_{v \in V} \frac{\gamma_v \delta^N_v(q)}{q_v + \frac{1}{N}}. \]
	Since $q \in U^N$, and in particular $m \leq q_v \leq M$ for every $v$, we obtain by using the equivalence of the $L_1$ and $L_2$ norms that
	\[ A \leq - c_1 N^a \norm{\delta^N(q)}{2}^2 + c_2 \norm{\delta^N(q)}{2} \]
	for some positive constants $c_1$ and $c_2$ that only depend on $n$, $m$, $M$ and $\gamma$. It is readily checked that the supremum of the function $x \mapsto c_2 x - c_1 N^a x^2$ is equal to $c^2_2 / (4 c_1 N^a)$, which gives $A \leq C N^{-a}$ as desired.
	
	Let us now control $B$. Computing the second derivative of $d^N$ gives
	\[ \partial^2_{v,v} d^N(q) = \frac{a \left( \delta^N_v(q) + a \pi^{Nq}(v) (1 - \pi^{Nq}(v)) \right)}{(q_v + \frac{1}{N})^2}. \]
	In particular, since $q \in U^N$ and $\lVert \zeta^v_\pm - q \rVert_\infty \leq 1/N$, we have
	\[ \left \lvert \partial^2_{v, v} d^N \left( \zeta^N_{\pm} \right) \right \rvert \leq \frac{1}{m^2} \]
	and so
	\[ B \leq \frac{1}{2 m N^{-(1-a)}} \left( 2 + N^{-a} \tmpsum(\gamma) \right) \]
	which gives $B \leq C N^{-(1-a)}$ for $N$ large enough, as desired.
\end{proof}

\color{black}


\begin{lem}\label{eqpoissc}
	We have
	\[ \E \left[ \sup_{0 \leq t \leq T \wedge T^N} \left \lvert \int_0^{t\wedge T^N} (L^N - L^N_{\mathrm{h}}) d^N (Q^N(s), \sigma^N(s))\d s \right \rvert \right] \to 0. \]
\end{lem}

\begin{proof}
	Since
	\begin{align*}
		(L^N - L^N_{\mathrm{h}}) d^N (q, \sigma) & = (L^{N,\sigma}_{\mathrm{s}} - L^N_{\mathrm{h}}) d^N (q)\\
		& = N^{a+1} \sum_{v \in V} \left( \sigma_v - \pi^{Nq}(v) \right) \ind_{q_v>0} \Delta^N_{-,v} d^N (q)
	\end{align*}
	Proposition~\ref{eqpoi} gives
	\begin{align*}
		\E \Bigg[ \sup_{0 \leq t \leq T \wedge T^N} \Bigg \lvert \int_0^{t\wedge T^N} (L^N & - L^N_{\mathrm{h}}) d^N (Q^N(s), \sigma^N(s))\d s \Bigg \rvert \Bigg]\\
		& \leq C N^{a+1} \max_{v \in V} \norm{\Delta^N_{-,v} d^N}{\infty, U}\dfrac{(\log N)^{3/2}}{N^{1/2}}\\
		& \hspace{10mm} + C N^{a+1} \max_{v,w} \norm{\partial_{w} \Delta^N_{-,v} d^N}{\infty, U} \dfrac{(\log N)^{3/2}}{N^{1-a}}.
	\end{align*}
	For $q \in U^N$ and $w \in V$ one can check that $\left \lvert \Delta^N_{-,v} d^N(q) \right \rvert, \left \lvert \partial_{w} \Delta^N_{-,v} d^N(q) \right \rvert \leq \frac{C}{N}$
	so that
	\begin{multline*}
		\E \left[ \sup_{0 \leq t \leq T \wedge T^N} \left \lvert \int_0^{t\wedge T^N} (L^N - L^N_{\mathrm{h}}) d^N (Q^N(s), \sigma^N(s))\d s \right \rvert \right] \leq C \dfrac{(\log N)^{3/2}}{N^{1/2-a}}.
	\end{multline*}
	Thus for $a < 1/2$ this bound indeed vanishes, which proves the result.
\end{proof}

\begin{lem}\label{lem:d}
	We have
	\[ \E \left[ \sup_{0 \leq t \leq T \wedge T^N} \left \lvert M^N_{d^N}(t) \right \rvert \right] \leq \frac{C}{N^{(1-a)/2}}. \]
\end{lem}

\begin{proof}
	Proceeding as in the proof of Lemma~\ref{lem:M} we obtain
	\begin{multline*}
		\E \left[ \sup_{0 \leq t \leq T \wedge T^N} M^N_{d^N}(t)^2 \right] \leq 4 N^{a+1}\cesp{}{\int_0^{T\wedge T^N} \sum_{v \in V} ( \Delta^N_{+,v} d^N (Q^N(s) )^2\d s}\\
		+ 4 N^{a+1} \cesp{}{\int_0^{T\wedge T^N} \sum_{v \in V}^n (\Delta^N_{-,v} d^N (Q^N(s))^2\d s}.
	\end{multline*} 
	The result then follows from the same Taylor expansion as in the proof of Lemma~\ref{prop2}.
\end{proof}

\color{black}

\section{Proof of main result} \label{Smain}
To prove Theorem \ref{main}, we will establish its equivalent for the stopped process $Q^N( \, \cdot \,\wedge T^N)$ using Gronwall's lemma. We then transfer the result on the stopped process to $Q^N$ using Lemma~\ref{lemma:m-M}.

\subsection{First step: $\tmpsum \circ Q^N(\, \cdot \, \wedge T^N) \Rightarrow S$}

The first step is to prove that $\tmpsum \circ Q^N(\, \cdot \, \wedge T^N) \Rightarrow S$ uniformly on $[0,T]$, which we do now. Starting from the definition of $L^{N, \sigma}_{\mathrm{s}}$ and using $\tmpsum(\lambda^N) = \textcolor{black}{1- N^{-a} \tmpsum(\gamma)}$ by~\eqref{eq:near-criticality} and $\sum_{v \in V} \sigma_v = 1 - \sigma_0$ we obtain
\[ L^{N, \sigma}_{\mathrm{s}} \tmpsum(q) = N^a\textcolor{black}{\tmpsum(\lambda^N)} - N^a \sum_{v \in V} \sigma_v \ind_{q_v > 0} = N^a \sigma_\zero + N^a \sum_{v \in V} \sigma_v \ind_{q_v = 0}\textcolor{black}{-\tmpsum(\gamma)}. \]
The semimartingale decomposition of $\tmpsum \circ Q^N$ and the fact that $$S(t) = S(0) + \mu \int_0^t S(s)^{-a} \d s\textcolor{black}{-\tmpsum(\gamma)t}$$ by definition of $S$ then leads to
\begin{multline} \label{eq:\tmpsum}
	\tmpsum(Q^N(t)) - S(t) = \tmpsum(Q^N(0)) - S(0) + \int_0^t \left( N^a \sigma^N_\zero - \frac{\mu}{S(s)^a} \right) \d s\\
+ \sum_{v \in V} \int_0^t \sigma^N_v(s)\ind_{Q^N_v(s)=0}\d s + M^N_{\tmpsum}(t).
\end{multline}
Define
\[ \varepsilon^N(t) = \tmpsum(Q^N(0)) - S(0) + \eta^N(t) + e^N(t) + h^N(t) + M^N_{\tmpsum}(t) \]
where
\[ \eta^N(t) = \int_0^{t \wedge T^N} \left ( \frac{1}{N^{-a} + \norm{Q^N(s) + 1/N}a^a} - \frac{1}{\norm{Q^N(s)}a^a} \right ) \d s, \]
\[ e^N(t) = \int_0^{t \wedge T^N} \left( \frac{1}{\norm{Q^N(s)}a^a} - \frac{\mu}{\tmpsum(Q^N(s))^a} \right) \d s \]
and
\[ h^N(t) = N^a \int_0^{t \wedge T^N} \left( \sigma^N_\zero(s)-\pi^{NQ^N(s)}(\zero) \right) \d s. \]
Since $Q^N_v(s) > 0$ for $t < T^N$, starting from~\eqref{eq:\tmpsum} and plugging in the above expressions, we obtain
\[ \tmpsum(Q^N(t \wedge T^N)) - S(t) = \varepsilon^N(t) + \mu \int_0^t \left( \frac{1}{\tmpsum(Q^N(s))^a} - \frac{1}{S(s)^a} \right) \d s. \]
Since $x \in [m,M] \mapsto x^{-a}$ is Lipschitz and all queue lengths are in $[m,M]$ before time $T^N$, we finally obtain
\[ \left \lvert \tmpsum(Q^N(t \wedge T^N)) - S(t) \right \rvert \leq \left \lvert \varepsilon^N(t) \right \rvert + C \int_0^t \left \lvert \tmpsum(Q^N(s \wedge T^N)) - S(s) \right \rvert \d s \]
and Gronwall's lemma implies
\begin{multline*}
	\sup_{0 \leq t \leq T} \left \lvert \tmpsum(Q^N(t \wedge T^N)) - S(t) \right \rvert\\
	\leq \left ( \left \lvert \tmpsum(Q^N(0)) - S(0) \right \rvert + \bar \eta^N + \bar e^N + \bar h^N + \sup_{0 \leq t \leq T \wedge T^N} \left \lvert M^N_{\tmpsum}(t) \right \rvert \right ) e^{CT}
\end{multline*}
with
\[ \bar \eta^N = \int_0^{T \wedge T^N} \left \lvert \frac{1}{N^{-a} + \norm{Q^N(s)+1/N}a^a} - \frac{1}{\norm{Q^N(s)}a^a} \right \rvert \d s, \]
\[ \bar e^N = \int_0^{T \wedge T^N} \left \lvert \frac{1}{\norm{Q^N(s)}a^a} - \frac{\mu}{\tmpsum(Q^N(s))^a} \right \rvert \d s \]
and
\[ \bar h^N = N^a \sup_{0 \leq t \leq T \wedge T^N} \left \lvert \int_0^t \left( \sigma^N_\zero(s) - \pi^{NQ^N(s)}(\zero) \right) \d s \right \rvert. \]
By assumption we have $\tmpsum(Q^N(0)) \to S(0)$ and so in order to prove the desired result $\tmpsum \circ Q^N(\, \cdot \, \wedge T^N) \Rightarrow S$ on $[0,T]$, we only have to prove that $\bar \eta^N, \bar e^N, \bar h^N$ and the martingale term vanish. The fact that $\bar \eta^N \Rightarrow 0$ comes directly from the fact that $Q^N(s) \in U$ for $s \leq T \wedge T^N$. The martingale term is handled with the exact same arguments as the previous martingale terms in Lemmas~\ref{lem:M} and~\ref{lem:d}, the proof is omitted. The next two lemmas show that the last two terms $\bar \epsilon^N$ and $\bar h^N$ also vanish.

\begin{lem}
	We have $\bar e^N \Rightarrow 0$.
\end{lem}

\begin{proof}
	\color{black}
	Since the process is stopped at $T^N$ and so all coordinates considered are bounded away from $0$, all the functions considered are Lipschitz and so according to Lemma~\ref{lem:diff-I} we have
	\[ \left \lvert Q^N_v(s)^a - \frac{\lambda^\infty_v}{\mu} s(Q^N(s)) \right \rvert \leq C d^\infty(Q^N(s))^{a/2}. \]
	Using the triangular inequality and the fact that $s(\lambda^\infty) = 1$, we thus obtain
	\[ \left \lvert \left \lVert Q^N(s) \right \rVert^a_a - \frac{1}{\mu} s(Q^N(s)) \right \rvert \leq C d^\infty(Q^N(s))^{a/2}. \]
	The convergence $\bar e^N \Rightarrow 0$ follows therefore readily from Proposition~\ref{prop:SSC} which implies that $d^\infty(Q^N(s)) \Rightarrow 0$ uniformly in $s \leq T \wedge T^N$.
\end{proof}

\begin{lem}
	We have $\E(\bar h^N) \to 0$.
\end{lem}

\begin{proof}
	Since $\sigma_\zero = \sum_{v \in V} \sigma_v$ and $\pi^q(\zero) = \sum_{v \in V} \pi^q(v)$ we have
	\[ \bar h^N \leq N^a \sum_{v \in V} \sup_{0 \leq t \leq T \wedge T^N} \left \lvert \int_0^t \left( \sigma^N_v(s)-\pi^{NQ^N(s)}(v) \right) \d s \right \rvert \]
	and so Proposition~\ref{lem:Poisson-1} with $f(q)=q$ implies that
	\[ \E(\bar h^N) \leq C \frac{(\log N)^{3/2}}{N^{1/2-a}}. \]
	As $a < 1/2$ we have the result.
\end{proof}

\subsection{Second step: proof of Theorem~\ref{main}}

\color{black}We now conclude the proof of Theorem~\ref{main}, so we have to control $Q^N_v(t) - q_v(t)$. The idea is to combine the convergence $s \circ Q^N (\cdot \wedge T^N) \Rightarrow s \circ q = S$ of the previous step, together with the state space collapse property $d^\infty \circ Q^N \Rightarrow 0$ of Proposition~\ref{prop:SSC}. Since $q(t) \in I$, we have $q_v(t) = (\lambda^\infty_v / \mu)^{1/a} S$ and so this leads us to write
\begin{multline*}
	\sup_{0 \leq t \leq T \wedge T^N} \left \lvert Q^N_v(t) - q_v(t) \right \rvert \leq \sup_{0 \leq t \leq T \wedge T^N} \left \lvert Q^N_v(t) - \left( \frac{\lambda^\infty_v}{\mu} \right)^{1/a} s(Q^N(t)) \right \rvert +\\
	\sup_{0 \leq t \leq T \wedge T^N} \left \lvert \tmpsum(Q^N(t)) - S(t) \right \rvert.
\end{multline*}
The second term vanishes by the first step, and so the first term vanishes as a consequence of the state space collapse property (combine Proposition~\ref{prop:SSC} and Lemma~\ref{lem:diff-I}). Thus we have proved that $Q^N(\cdot \wedge T^N) \Rightarrow q$.
\color{black}

Let us now remove the localization and prove that $Q^N \Rightarrow q$. In order to do so, it is enough to show that $\P(T^N \geq T) \to 1$. By definition of $T^N$, we have
\[ \left \lVert Q^N(T^N) - q(T^N) \right \rVert_1 \geq \frac{m}{2}. \]
Since $T^N \wedge T = T^N$ in the event $\{T^N \leq T\}$, this entails
\[ \P \left( T^N \leq T \right) \leq \P \left( \left \lVert Q^N(T^N \wedge T) - q(T^N \wedge T) \right \rVert_1 \geq \frac{m}{2} \right). \]
Since we have proved that $Q^N(\, \cdot \, \wedge T^N) \Rightarrow q$ uniformly on $[0,T]$, the previous probability vanishes. This concludes the proof of Theorem~\ref{main}.

\color{black}

\section{Extensions and directions for future research} \label{sec:openings}

\color{black}

\subsection{Beyond $a < \frac{1}{2}$} \label{sub:1/2}

Proposition~\ref{eqpoi} shows that the averaging approximation~\eqref{eq:SAP} holds for $a < 1$. This is in line with Lemma~\ref{lemma:lambda} which shows that the mixing time of the fast process is of the order of $N^a$: since the typical time scale of the slow process is $N$, the condition $a < 1$ reflects that the fast process evolves much faster than the slow process, which is the condition expected for homogenization to hold.

However, our condition in Theorem~\ref{main} is the more stringent condition $a < 1/2$. To see why this condition appears, consider the following semimartingale decomposition of $Q^N$:
\begin{multline*}
	Q^N_v(t) - Q^N_v(0) = N^a \int_0^t \left( \lambda^\infty_v - \pi^{NQ^N(s)}(v) \right) \d s\\
	+ \text{(martingale term)} + N^a \int_0^t \left( \sigma^N_v(s) - \pi^{NQ^N(s)}(v) \right) \d s.
\end{multline*}
The martingale term can be shown to vanish for $a < 1$, but we see that in order for the first term to also vanish we would need to show that the integral on the second line is $o(N^{-a})$: Proposition~\ref{eqpoi} shows that this term is $O(1/N^{1/2} + 1/N^{1-a})$ and so although it is $o(1)$ for $a < 1$, in order to have it $o(N^{-a})$ we need to assume that $a < 1/2$. Whether Theorem~\ref{main} continues to hold for $1/2 < a < 1$ constitutes in our view an interesting open problem, which also testifies to the difficulty of proving fully coupled stochastic averaging principles even in seemingly simple cases.


\color{black}

\subsection{Two other scalings} \label{sub:other-scalings}

Keeping $\varepsilon > 0$ as the distance to the stability region, we now discuss what happens on different space scales than the scale $N = \varepsilon^{-1/a}$ studied so far. To be more precise, we continue to consider arrival rates $\lambda$ given by
\[ \lambda = \lambda^\infty - \varepsilon \gamma \]
with $\tmpsum(\lambda^\infty) = 1$, but now we consider the queue length process on the space scale $N = \varepsilon^{-1/a'}$ with $a' > 0$. Let $N^b$ be the general time scale, and so consider the scaled processes
\[ Q^N(t) = \frac{1}{N} Q(N^b t) \ \text{ and } \ \sigma^N(t) = \sigma(N^b t), \ t \geq 0. \]
Assume for a moment that the stochastic averaging principle and state space collapse continue to hold: thus, determining the asymptotic behavior of $Q^N$ reduces (at least informally) to understanding the asymptotic behavior of $\tmpsum \circ Q^N$ under the homogenized dynamic. With the considered scaling, the homogenized generator is given by
\begin{multline*}
	L^N_\mathrm{h} f(q) = N^b \sum_{v \in V} \lambda^N_v \left( f \left( q + \frac{e^v}{N} \right) - f \left( q \right) \right)\\
	+ N^b \sum_{v \in V} \pi^{Nq}(v) \ind(q_v > 0) \left( f \left( q - \frac{e^v}{N} \right) - f \left( q \right) \right)
\end{multline*}
and so for $q > 0$ we have
\begin{align*}
	L^N_\mathrm{h} \tmpsum(q) & = N^{b-1} \sum_{v \in V} \lambda^N_v - N^{b-1} \sum_{v \in V} \pi^{Nq}(v)\\
	& = N^{b-1} \left( 1 - \tmpsum(\gamma) N^{-a'} \right) - N^{b-1} \left( 1 - \pi^{Nq}(0) \right).
\end{align*}
We have
\[ \pi^{Nq}(0) = \frac{1}{1 + \sum_{v \in V} (N q_v + 1)^a} \approx \frac{1}{N^a \norm{q}a^a} \]
and since the state space collapse assumption entails $\norm{q}a^a = \tmpsum(q)^a/\mu$, we obtain
\[ L^N_\mathrm{h} \tmpsum(q) \approx - N^{b-1-a'} \tmpsum(\gamma) + N^{b-1-a} \mu \tmpsum(q)^{-a}. \]
We see that except when $a = a'$, which is the case studied so far, we cannot have both terms contributing in the limit: one dominates the other. In some sense, the space scale $N = \varepsilon^{-1/a}$ is the only one where we see in the limit at the same time the influence of the idleness induced by the distributed scheduling and the asymptotic drift term arising from the pre-limit processes being near-critical. More precisely, two cases arise:
\begin{description}
	\item[Case $a' < a$:] this is the space scale on which the near-criticality assumption dominates, the idleness induced by the distributed scheduling has no impact. In this case, the right time-scale is $b = 1+a'$ and $Q^N \Rightarrow q$ with $q(t) \in I$ for all $t \geq 0$, and $\tmpsum \circ q$ solution to $\dot x = -\tmpsum(\gamma) \ind(x > 0)$, i.e., $\tmpsum(q(t)) = (\tmpsum(q(0)) - \tmpsum(\gamma) t)_+$;
	\item[Case $a' > a$:] this is the reversed case: one only sees the idleness induced by the distributed scheduling, the limit is the same as the one obtained in Theorem~\ref{main} with $\gamma = 0$. In this case, the right time-scale is $b = 1+a$ and $Q^N \Rightarrow q$ with $q(t) \in I$ for all $t \geq 0$ and $\tmpsum \circ q$ solution to $\dot x = \mu x^{-a}$.
\end{description}
Except for controlling $\tmpsum \circ Q^N$ after (potentially) hitting $0$ in the case $a' < a$, these results can be established by making appropriate changes in the arguments developed above for $a' = a$. Actually, only minor changes are needed along the way. When $a' < a$ and $\tmpsum(\gamma) > 0$, which is the only case where $s \circ q$ hits $0$ in finite time, namely $s(q(0)) / s(\gamma)$, $\tmpsum \circ Q^N$ can be controlled after time $\tmpsum(q(0)) / \tmpsum(\gamma)$ by coupling arguments that will be developed in~\cite{Castiel:+}.

\subsection{Interchange of limits}

Heavy traffic results are often investigated as a means to establish convergence of stationary distributions according to the well-known interchange of limits argument presented schematically in Figure~\ref{fig:interchange}. In our case, $Q^N$ admits a stationary distribution $Q^N(\infty)$ in the subcritical case $\tmpsum(\gamma) > 0$, and in this case we have $Q^N \Rightarrow q$ with $q(t) \in I$ for all $t \geq 0$ and $\tmpsum \circ q$ solution to $\dot x = \mu x^{-a} - \tmpsum(\gamma)$. Although we do not know how to solve this equation explicitly, it is readily seen that, when $\tmpsum(\gamma) > 0$, the solution to this ODE converges to $\beta := (\mu / \tmpsum(\gamma))^{1/a}$ as $t \to \infty$. At first, the interchange of limits of arguments seems therefore to suggest that $Q^N(\infty) \Rightarrow \beta$. The result being deterministic, this would be a rather unusual heavy traffic result. However, we do not know whether this reasoning applies because of the following argument.

\begin{figure}[tbp]
	\centering
		\begin{tikzpicture}[scale=.5]
			\node (1) at (0,0) {$X^N(t)$};
			\node (2) at (5,0) {$X(t)$};
			\node (3) at (0,-5) {$X^N(\infty)$};
			\node (4) at (5,-5) {$X(\infty)$};
			\draw [->] (1) -- (2) node [midway, below] {$N \to \infty$};
			\draw [->] (1) -- (3) node [midway, left] {$t \to \infty$};
			\draw [->] (2) -- (4) node [midway, right] {$t \to \infty$};
			\draw [->, dashed] (3) -- (4) node [midway, above] {?} node [midway, below] {$N \to \infty$};
		\end{tikzpicture}
	\caption{\color{black}Illustration of the interchange of limits argument: if $X^N \Rightarrow_N X$ and $X(t) \Rightarrow_t X(\infty)$, then provided technical assumptions (typically, tightness of $(X^N(\infty))$) we have $X^N(\infty) \Rightarrow_N X(\infty)$.}
	\label{fig:interchange}
\end{figure}
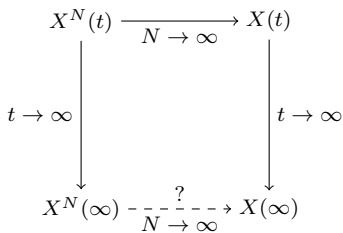

If we start at $Q^N(0)$ with $\tmpsum(Q^N(0)) \to \beta$, then $Q^N \Rightarrow q$ with $q$ a constant function, which suggests to look at $Q$ on a faster time scale than $N^{1+a}$. Indeed, it is thus conceivable that $Q$ scaled differently converges to a non-constant, possibly random, limit. More precisely, the fact that $Q^N \Rightarrow q$ with $q$ a constant function opens the possibility that $\widetilde Q^N \Rightarrow X$ with $X$ a diffusion process, where
\[ \widetilde Q^N(t) = \frac{1}{N} Q(N^b t), \ t \geq 0, \]
for some $b > 1+a$. In this case, the interchange of limits would suggest that $Q^N(\infty) \Rightarrow X(\infty)$, provided $X$ has a stationary distribution.

This argument is plausible because this is typically what happens when considering the near-critical case. For instance, a near-critical $M/M/1$ queue in the fluid regime converges to a constant function, but in the diffusive scale it converges to a positive recurrent Brownian motion. If one were only to consider the fluid limit and naively apply the interchange of limits principle, one would be led to conclude that the stationary distribution converges to a constant, which is not the case. Thus, the asymptotic behavior of $Q^N$ stationary distribution constitutes in our view an intriguing open question.

\subsection{Beyond a complete interference graph}

What makes the case of a complete interference graph tractable is that all queues are of the same order of magnitude and remain away from~$0$ at all times, i.e., all coordinates are scaled by~$N$ and the limiting process $q$ satisfies $\inf_{t \geq 0} q_v(t) > 0$ for every $v \in V$. We believe that the techniques developed in the present paper can be applied beyond the case of the complete interference graph as long as this property holds. For instance, they should be applicable to a square interference graph with four nodes $1,2,3,4$ and edges $(1,2)$, $(2,3)$, $(3,4)$ and $(1,4)$ and equal arrival rates $\lambda_v = 1/4$ at all nodes. Note that in this case, the stability condition is $\max(\lambda_1, \lambda_3) + \max(\lambda_2, \lambda_4) < 1$ so all $\lambda$'s equal to $1/2$ is indeed critical.

However, the fact that all queues remain positive at all times now depends on the underlying interference graph and also on the arrival rates. If we take the above square interference graph with $\lambda_1 = \lambda_2 = \lambda_3 = 1/2$ but $\lambda_4 < 1/2$, then we believe that queue $4$ will remain at $0$. In this case, our techniques can no longer apply, especially the localization arguments that need all queue lengths to be bounded away from $0$. We believe that in such cases, subtle behavior can arise and calls for new ideas.

To give a flavor of the kind of possible new behavior, consider three nodes on a line: the interference graph has three nodes $1,2,3$ and two edges $(1,2)$ and $(2,3)$. In this case, the 'outer' nodes $1$ and $3$ compete against the 'middle' node $2$ to access the channel. The two maximal independent sets are $\{1,3\}$ and $\{2\}$ with respective weight, for the Glauber dynamics,
\[ \pi^q(\{1,3\}) = \frac{(q_1 q_3)^a}{1 + q_1^a + q_2^a + q_3^a + (q_1 q_3)^a} \] 
and
\[ \pi^q(\{2\}) = \frac{q_2^a}{1 + q_1^a + q_2^a + q_3^a + (q_1 q_3)^a}. \]
In the critical and symmetric case $\lambda_1 = \lambda_2 = \lambda_3 = 1/2$, we must have $\pi^q(\{1,3\}) = \pi^q(\{2\}) = 1/2$ in order for service to match arrivals, which imposes $q_1 q_3 = q_2$. This relation imposes a constraint on the product $q_1 q_3$ but not on the individual queues $q_1$ and $q_3$. In particular, if $q_2$ is of the order of $N$, then $q_1$ and $q_3$ will be much smaller, say $\sqrt N$ each. Different queues may thus live on different space scales, which suggests the necessity for a multiscale analysis.

\color{black}

\end{document}